\documentclass[12pt]{article}

\usepackage{graphicx}
\usepackage{url}            
\usepackage{amsmath}
\usepackage{amsbsy,amssymb}
\usepackage{algorithm}
\usepackage{algorithmic}
\usepackage{multicol}
\usepackage{sidecap}
\usepackage{flexisym}
\usepackage{cleveref}
\usepackage{color}

\def\sgn{\mathop{\rm sgn}\nolimits}

\def\prox{\mathop{\rm prox}\nolimits}
\def\amp{\mathop{\:\:\,}\nolimits}
\def\Real{\mathop{\mathbb{R}}\nolimits}

\DeclareMathOperator*{\argmin}{argmin}

\def\sup{\mathop{\rm sup}\nolimits}

\newcommand{\bb}{\boldsymbol{b}}

\newcommand{\bs}{\boldsymbol{s}}

\newcommand{\bu}{\boldsymbol{u}}
\newcommand{\bv}{\boldsymbol{v}}

\newcommand{\bx}{\boldsymbol{x}}
\newcommand{\by}{\boldsymbol{y}}
\newcommand{\bz}{\boldsymbol{z}}
\newcommand{\bA}{\boldsymbol{A}}
\newcommand{\bB}{\boldsymbol{B}}

\newcommand{\bV}{\boldsymbol{V}}

\newcommand{\bY}{\boldsymbol{Y}}
\newcommand{\bZ}{\boldsymbol{Z}}

\newcommand{\bepsilon}{\boldsymbol{\epsilon}}

\usepackage{fullpage}
\usepackage{setspace}
\usepackage[authoryear]{natbib}

\usepackage{cleveref}

\usepackage{authblk}

\newcommand{\qed}{$\square$}

\allowdisplaybreaks

\newtheorem{proposition}{Proposition}

\newtheorem{remark}{Remark}

\newtheorem{proof}{Proof}

\title{A unified analysis of convex and non-convex $\ell_p$-ball projection problems}
\author[1]{Joong-Ho Won}
\author[2]{Kenneth Lange}
\author[3]{Jason Xu}
\affil[1]{Seoul National University}
\affil[2]{University of California, Los Angeles}
\affil[3]{Duke University}
\date{}

\begin{document}

\maketitle
\abstract{
	The task of projecting onto $\ell_p$ norm balls is ubiquitous in statistics and machine learning, yet the availability of actionable algorithms for doing so is largely limited to the special cases of $p = \left\{ 0, 1,2, \infty \right\}$. In this paper, we introduce novel, scalable methods for projecting onto the $\ell_p$ ball for general $p>0$. For $p \geq1 $, we solve the univariate Lagrangian dual via a dual Newton method. We then carefully design a bisection approach for $p<1$, presenting theoretical and empirical evidence of zero or a small duality gap in the non-convex case. The success of our contributions is thoroughly assessed empirically, and applied to large-scale regularized multi-task learning and compressed sensing.
 }

\section{Introduction}

The goal of this paper is to 
develop and analyze efficient algorithms
for projecting a point $\by$ in Euclidean space $\Real^d$ onto an $\ell_p$ ``norm ball'' of radius $r$. Projection seeks the closest point in the ball, solving the problem
\begin{equation}\label{eqn:projection}\tag{P}
	\min_{\bx\in\Real^d} \frac{1}{2}\|\bx - \by\|_2^2 \quad \text{subject to} \quad \|\bx\|_p \leq r
	,
\end{equation}
where $\lVert \bv \rVert_p = (\sum_{i=1}^d  | v_i |^p )^{1/p}$ denotes the $\ell_p$ ``norm'' of a vector $\bv \in\Real^d$.
If the power $p\ge 1$, then $\|\bv\|_p$ is a proper norm, and \eqref{eqn:projection} is a convex optimization problem. Otherwise,   $\|\bv\|_p$ defines only a quasi-norm, and the problem becomes non-convex.
Throughout, we maintain the term ``norm'' even if $p\in[0, 1)$, for nomenclatural convenience. Writing the norm-ball constraint set as $rB_p$, the solution to  \eqref{eqn:projection} is the projection denoted $P_{rB_p}(\by)$.

Projecting onto $\ell_p$ norm balls plays a key role in some of the most prominent inverse problems in machine learning, signal processing, and statistics. The canonical setting of minimizing a measure of a fit subject to a constraint on solution complexity 
measured by a norm  \citep{candes2005,donoho2006}
arises, for instance, in compressed sensing.
If $\phi(\bx;\by)$ measures the goodness of fit of model parameter $\bx$ to given data $\by$, then 
the problem is stated
\begin{equation}\label{eq:gen_constrained}
\min_{\bx\in\Real^d} \phi(\bx; \by) \quad \text{subject to} \quad \|\bx\| \leq r.
\end{equation}
Among the popular choices of the norm $\|\cdot\|$ are the $\ell_p$ norms
and $\ell_{1,p}$ mixed norms denoted $\|\bx\|_{1,p}=\sum_{g=1}^G\|\bx_g\|_p$, where $\bx$ is partitioned into subvectors indexed by $g$.
The $\ell_2$ norm is used in Tikhonov regularization or ridge regression. 
Bridge regression \citep{fu1998penalized} uses an $\ell_p$ norm optimized over $p \in (1, 2)$.
The $\ell_1$ norm is the most widely used regularizer in sparse learning \citep{tibshirani2015}. 
In compressed sensing, exact signal recovery can be accomplished with fewer measurements under non-convex $\ell_p$ norms ($0<p<1$) compared to the $\ell_1$ norm \citep{wang2011performance}.

Whatever the choice of $p$, algorithms for solving problem \eqref{eq:gen_constrained} routinely require 
solving projection \eqref{eqn:projection} as a subproblem.
For instance, if $\phi$ is smooth, projected gradient descent iterates via the sequence 
 $\bx_{k+1} = P_{rB_p}[ \bx_k - \eta_k \nabla \phi (\bx_k)]$.
In multi-task learning \citep{argyriou2008,sra2012,vogt2012}, 
the use of the $\ell_{1,2}$ (``group lasso'') \citep{yuan2006model,meier2008}
or $\ell_{1,\infty}$ norms \citep{liu2009,quattoni2009} is popular, 
while
an $\ell_{1,p}$ norm with
$p\in (1, \infty]$ chosen in a data-adaptive fashion has been shown to significantly improve performance \citep{zhang2010}.
Projection onto an $\ell_{1,p}$ norm ball for solving problem \eqref{eq:gen_constrained} separates into $G$ instances of \eqref{eqn:projection}.
Hence efficiently solving \eqref{eqn:projection} for various values of $p>0$ is instrumental to solving problems of type \eqref{eq:gen_constrained}.

However, algorithms for solving \eqref{eqn:projection} are quite limited to a few special values of $p$.
Projecting onto the $\ell_2$ or $\ell_\infty$ ball is trivial, and fast algorithms are available for the $\ell_1$ ball 
\citep{duchi2008,condat2016fast}. 
For general convex settings ($p > 1$), 
problem \eqref{eqn:projection} has been studied mostly as a subproblem of $\ell_{1,p}$ mixed norm regularization \citep{liu2010efficient,zhou2015ell1p,barbero2018}.
A common approach is 
to formulate \eqref{eqn:projection} as an unconstrained problem
\begin{equation}\label{eqn:simple}\tag{\ref{eqn:projection}\textprime}
	\min_{\bx\in\Real^d} f_0(\bx) + \iota_{rB_p}(\bx)
	,
\end{equation}
where $f_0(\bx)=\frac{1}{2}\|\bx-\by\|_2^2$, 
and $\iota_C$ is the $0$/$\infty$ indicator function of set $C$.
Problem \eqref{eqn:simple} can be solved via its Fenchel (or Lagrange) dual 
\begin{equation}\label{eqn:fenchel}\tag{F}
	\min_{\bz}\frac{1}{2}\|\bz-\by\|_2^2 + r\|\bz\|_q,
	\quad
	\text{where}
	\quad
	\frac{1}{p}+\frac{1}{q} = 1
	,
\end{equation}
whose solution 
is the proximal operator of the dual norm  $\bz^{\star}=\prox_{r\|\cdot\|_q}(\by)$. 
The solution $P_{rB_p}(\by)$ to \eqref{eqn:projection} is then recovered by Moreau's decomposition $\by=P_{rB_p}(\by) + \bz^{\star}$. 
The technical report 
\citet{liu2010efficient} explored some properties of problem \eqref{eqn:fenchel}, 
and proposed a double-bisection method implemented in the popular software package SLEP \citep{liu2011slep}.
Recently, Barbero and Sra \cite{barbero2018} proposed solving \eqref{eqn:fenchel} via the projected Newton method \citep{bertsekas1982projected}.
The major difficulty with \eqref{eqn:fenchel} is its nonsmoothness. As we will detail in \Cref{sec:review}, the analysis of \citet{liu2010efficient} entails opaque auxiliary functions, not to mention slowness and poor scalability of the double-bisection method. 
As discussed in the sequel,
the method of Barbero and Sra \cite{barbero2018} suffers from numerical instability when $p$ is large.
In the non-convex regime, available projection methods are limited. 
\citet{bahmani2013} studies basic theoretical properties of  projected points.
In \citet{das2013non}, an exhaustive search is attempted. More recent efforts include \citet{chen2019outlier} and \citet{yang2021towards}.

In this paper, we 
provide a unified treatment of both convex and non-convex instances of  the projection problem \eqref{eqn:projection}. 
Our approach is to reformulate \eqref{eqn:projection} using the $p$th power as 
\begin{equation}\label{eqn:notsimple}\tag{\ref{eqn:projection}\textprime\textprime}
	\min_{\bx\in\Real^d} f_0(\bx) 
	\quad
	\text{subject to}
	\quad
	\frac{1}{p}(\Vert\bx\Vert_p^p - r^p) \leq 0
	.
\end{equation}
Then, the Lagrange dual of \eqref{eqn:notsimple} is
\begin{equation}\label{eqn:dual}\tag{D}
	\max_{\mu\ge 0}~g(\mu) \triangleq \inf_{\bx\in\Real^d} \mathcal{L}(\bx,\mu), 
	\quad
	\text{where}
	~~
	\mathcal{L}(\bx,\mu) = f_0(\bx) + \frac{\mu}{p}\left(\sum_{i=1}^d |x_i|^p - r^p\right) .
\end{equation}
If $\bx^{\star}(\mu)$ minimizes $\mathcal{L}(\bx,\mu)$ and $\mu^{\star}$ maximizes $g(\mu)$, then $P_{rB_p}(\by)=\bx^{\star}(\mu^{\star})$.
Compared to the Fenchel dual \eqref{eqn:fenchel}, the present formulation \eqref{eqn:dual} has three key advantages: 1) the Lagrangian $\mathcal{L}$ is separable in $\bx$; 2) the dual objective function $g$ is twice continuously differentiable if $p>1$, and hence is amenable to Newton methods for maximization; 3) it is \emph{univariate} and well-defined even when $0<p<1$. 
In the latter setting, minimization of the Lagrangian $\mathcal{L}(\bx, \mu)$ with respect to $\bx$, namely evaluation of $g(\mu)$, is relatively well-studied under the name of non-convex $\ell_p$ regularization \citep{marjanovic2012,xu2012,chartrand2016,yukawa2016,hu2017group}, compared to problem \eqref{eqn:projection}.
Since $g(\mu)$ is concave regardless of $p$, in principle any univariate maximization strategy for $g$ can be used to solve \eqref{eqn:dual}.

Though strong duality is not guaranteed when $0<p<1$,  we show that a carefully designed bisection method yields accurate solutions to \eqref{eqn:projection} with very small duality gaps. For convex cases, we show that our Newton method achieves a quadratic rate of convergence, with no projection onto the set $\{\mu: \mu \ge 0\}$ needed.
Fast convergence is paramount because \eqref{eqn:projection} is commonly used as a building block within iterative algorithms for more complex tasks, such as problem \eqref{eq:gen_constrained}.
The success of our methods hinges on fast and accurate evaluation of the associated univariate proximal map, which we closely analyze in the following sections. 
Our primary contribution is recognizing the value of formulation \eqref{eqn:dual} together with carefully executed analyses of the proximal map for both convex and non-convex cases.

The paper is organized as follows.
After analyzing the properties of \eqref{eqn:projection}, \eqref{eqn:dual}, and the associated proximal maps in \Cref{sec:proximal}, we propose a dual Newton method for $p>1$ and establish its convergence rate in \Cref{sec:convex}.
Next, \Cref{sec:nonconvex} details the viable alternative of bisection for non-convex cases.
Related methods are discussed in Section \ref{sec:review}.
In \Cref{sec:performance}, the algorithms are thoroughly assessed via simulation. Together they comprise a suite that allows for successful projection onto general $\ell_p$ balls. The empirical study also illustrates the key role of such projections within algorithms such as projected and proximal gradient \citep{Beck2009}, applied to multi-task learning and compressed sensing. 

 \section{Evaluating dual objective via associated proximal map}\label{sec:proximal}
\iffalse
\subsection{Basic properties}
We begin with a few 
observations when $\by$ is external to the ball $rB_p$ \citep{bahmani2013}: 
\begin{enumerate}
	\item The projected point $\bx$ is a boundary point in the sense that $\|\bx\|_p=r$. 
	\item The components of $\bx$ have the same signs as the corresponding components of $\by$. 
	\item No component $x_i$ of $\bx$ can satisfy $|x_i| > |y_i|$, and if $y_i=0$, then $x_i=0$.
	\item No two components $x_i$ and $x_j$ of $\bx$ can satisfy $|x_i| > |x_j|$ when $|y_i| < |y_j|$. 
	\item We can take the radius $r$ of the ball to be $1$. Indeed, if $r^{-1}\bx$ solves the corresponding problem for the unit ball and the external point $r^{-1}\by$, then $\bx$ solves the original problem. 
\end{enumerate}		
Henceforth, we take $\by > {\bf 0}$ (denoting elementwise inequality), $r=1$, and $\|\by\|_p>1$ 
without losing generality.
\fi

\subsection{Basic properties of the $\ell_p$-ball projection}
In this section, we study how to represent and evaluate the dual objective $g(\mu)$ of problem \eqref{eqn:dual} in terms of the associated univariate proximal maps.
We begin with a few simple 
observations when $\by$ is external to the ball $rB_p$, given in \citep{bahmani2013}:
\footnote{These properties have emerged in the context of studying theoretical properties of projected gradient descent for $\ell_p$-norm constrained least squares (problem \eqref{eq:gen_constrained} with $\phi(\bx,\by)=\frac{1}{2}\|\by - \bA\bx\|_2^2$). However, no actual algorithm for $\ell_p$-ball projection is provided in \citet{bahmani2013}.}
\begin{enumerate}
	\item The projected point $\bx$ is a boundary point in the sense that $\|\bx\|_p=r$. 
	\item The components of $\bx$ have the same signs as the corresponding components of $\by$. 
	\item No component $x_i$ of $\bx$ can satisfy $|x_i| > |y_i|$, and if $y_i=0$, then $x_i=0$.
	\item No two components $x_i$ and $x_j$ of $\bx$ can satisfy $|x_i| > |x_j|$ when $|y_i| < |y_j|$. 
	\item We can take the radius $r$ of the ball to be $1$. Indeed, if $r^{-1}\bx$ solves the corresponding problem for the unit ball and the external point $r^{-1}\by$, then $\bx$ solves the original problem. 
\end{enumerate}		
Henceforth, we take $\by > {\bf 0}$ (denoting elementwise inequality), $r=1$, and $\|\by\|_p>1$ 
without loss of generality.

\subsection{Univariate proximal map for problem \eqref{eqn:dual}}
As stated, evaluating the dual objective $g(\mu)$ in problem  \eqref{eqn:dual} requires minimizing the Lagrangian $\mathcal{L}(\bx,\mu)$ over $\bx$ (with $r=1$). The latter problem is equivalent to the $\ell_p$-regularized least squares problem
\[
	\min_{\bx} \frac{1}{2}\|\bx - \by\|_2^2 + \frac{\mu}{p}\|\bx\|_p^p
	,
\]
in which the $p$th power of the $\ell_p$ norm satisfies the triangle inequality when $0<p<1$ \citep{chartrand2016}.
For any $p$, this problem is separable in the coordinates of $\bx=(x_1, \dotsc, x_d)$, and it suffices to minimize the univariate function
\begin{equation}\label{eqn:proximal}
	f_\mu(x; y) \triangleq
	\frac{1}{2}(x-y)^2 + \frac{\mu}{p}|x|^p 
\end{equation}
for each $y=y_i$.
If we define $s_p(x) \triangleq \frac{1}{p}|x|^p$,
then the minimizer of $f_{\mu}(x; y)$ is just the proximal map $\prox_{\mu s_p}(y)$ of $\mu s_p$, 
that is,
\[
    \prox_{\mu s_p}(y) \triangleq \argmin_x \left\{\frac{1}{2}(x-y)^2 + \frac{\mu}{p}|x|^p \right\}
\]
for $\mu \geq 0$.
If $p \geq 1$, then $\prox_{\mu s_p}(y)$ is unique.
However, when $p < 1$, it may be set-valued; see \Cref{sec:prox:nonconvex} for details.

\subsection{The univariate dual}\label{sec:proximal:objective}

From the discussion in the preceding section, the dual objective $g(\mu)$ in \eqref{eqn:dual} can be written
\begin{equation}\label{eqn:dualobjective} 
	g(\mu) = \frac{1}{2}\sum_{i=1}^d[y_i-x_i^{\star}(\mu)]^2 + \frac{\mu}{p}\sum_{i=1}^d|x_i^{\star}(\mu)|^p - \frac{\mu}{p} 
\end{equation}
where $x_i^{\star}(\mu) = \prox_{\mu s_p}(y_i)$.
Even though the latter proximal map is set-valued when $0<p<1$, 
$g(\mu)$ is always single-valued, 
since any element of the set $\bx^{\star}(\mu)=x_1^{\star}(\mu) \times \dotsb \times x_d^{\star}(\mu)$ globally minimizes $\mathcal{L}(\bx,\mu)$. 

We now examine the domain of $g$. 
A nonzero minimizer of $f_{\mu}(x_i; y_i)$ in formula \eqref{eqn:proximal}, i.e., if $x_i=\prox_{\mu s_p}(y_i)> 0$, satisfies the stationary condition
\begin{eqnarray}\label{stationarity_eq}
0 & = & x_i-y_i + \mu x_i^{p-1}
.
\end{eqnarray} 
(Recall that $y_i > 0$ is assumed without loss of generality.)
Now, multiply equation \eqref{stationarity_eq} by $x_i$ and sum over all $i$ (so that the zero minimizer $x_i=0$ is allowed even if equation \eqref{stationarity_eq} is satisfied). These actions produce 
\begin{equation}\label{eqn:stationarycond}
0 = \|\bx\|_2^2-\by^T\bx+\mu\|\bx\|_p^p
,
\end{equation}
which in turn yields $\mu = \by^T\bx-\|\bx\|_2^2$ for $\|\bx\|_p=1$ 
(recall that we assume $r = 1$ and $\Vert\by\Vert_p > 1$; the optimum occurs at the boundary of the constraint set). 
Note a generalization of equation \eqref{eqn:stationarycond} is studied in \citet{chen2013optimality,lu2014iterative} for $\ell_p$-regularization problems and their extensions. 
Note that $\mu$ so defined is automatically nonnegative; 
$\mu$ is positive so long as $0 < |x_i| < |y_i|$ for some $i$. 
The formula for $\mu$ also yields an upper bound through H{\"o}lder's inequality: if $p \ge 1$, $\,\mu  \le  \by^T\bx  \le  \|\by\|_q \|\bx\|_p = \|\by\|_q$, 
where $1/p+1/q=1$.
If $p \in (0, 1)$, then $\mu \le \by^T\bx \le \|\by\|_{\infty}\|\bx\|_1 \le \|\by\|_{\infty}$, since $\|\bx\|_p \le 1$ implies $\|\bx\|_1 \le 1$.
Thus, it suffices to maximize $g(\mu)$ on $[0, \|\by\|_q]$, where $\frac{1}{p}+\frac{1}{q}=1$ or $q=\infty$.

\subsection{Analysis and computation of the proximal map}\label{sec:prox}
\subsubsection{Convex case ($p > 1$)}\label{sec:prox:convex}
\paragraph{Newton's method}
The objective 
$f_\mu(x;y) = \frac{1}{2}(y-x)^2 + \mu s_p(x)$ 
of \eqref{eqn:proximal} is twice continuously differentiable on $[0, y]$ for $p>1$. 
Thus, $\prox_{\mu s_p}(y)$ can be evaluated 
via Newton's method. The updates amount to
\begin{equation}\label{newtonprox2} 
x_{n+1} = \frac{y+\mu(p-2)|x_{n}|^{p-1}\sgn(x_n)}
{1+\mu (p-1)|x_{n}|^{p-2}} 
 = \frac{{y}/{|x_{n}|^{p-2}}+\mu(p-2)x_{n}}
{{1}/{|x_{n}|^{p-2}}+\mu(p-1)}. 
\end{equation}
Whenever $p \ge 2$ and $y \ge 0$, \emph{all iterates remain in $[0,y]$}, and convergence
is guaranteed as the derivative $f_{\mu}'(x; y) = x-y+\mu x^{p-1}$
is convex in $x\ge 0$. Numerical overflows are still a consideration for very large values of $p$. Both forms in (\ref{newtonprox2}) are pertinent in this regard: the latter is numerically stable when $|x_n|> 1$ while the former is preferable when $|x_n|<1$. 
If the quantity $\mu |y|^{p-1}$ falls below machine precision, it is safe to set $\prox_{\mu s_p}(y)=y$.
When $1 < p < 2$, the iterates $x_{n}$ defined above may fall below 0. 
To avoid this, we may appeal to Moreau's decomposition that tells us that we may equivalently compute 
\[
	y - \mu \prox_{\mu^{-1}s_p^*}(\mu^{-1}y),
\]
where 
\[
	s_p^*(z)=\sup_x\{xz - s_p(x)\}= \frac{1}{q}|y|^q= s_q(y) 
	,
	\quad
	q = \frac{p}{p-1} \] 
is the Fenchel conjugate of $s_p$. Since $q>2$, applying Newton's method on $s_q$ avoids the issue.
Again numerical caution is required when $\mu^{-1}|\mu^{-1}y|^{q-1}=\mu^{-1}|\mu^{-1}y|^{1/(p-1)}$ nearly vanishes,  for in this case $x^\star(\mu) = \prox_{\mu s_p}(y) = y - \mu\prox_{\mu^{-1}s_q}(\mu^{-1}y) \approx y - y = 0$, but $f_{\mu}'(0; y) = -y \neq 0$.
Even though $x^\star(\mu)$ is very close to zero beyond numerical precision, $|x^\star(\mu)|^{p-1}$ may be close to 1 when $p \to 1$.
It is therefore safe to set 
\begin{equation}\label{eqn:safe}
|x^\star(\mu)|^{p-1} = |\mu^{-1}y|
\end{equation}
and $x^\star(\mu) = \sgn(y) |\mu^{-1}y|^{1/(p-1)}$.

\paragraph{Initial point}
Toward computing the proximal map,
we suggest initializing $x$ by $x_0=\max\{1, \allowbreak \min(y, [y + \mu(p - 2)]/[1 + \mu(p - 1)] )\}$ based on the following observation. When $p>1$ is very large, 
the behavior of the term $\mu |x|^{p-1}$ radically
changes from near $0$ to near $\infty$ as we pass from $|x|<1$ to $|x|>1$, the value $\prox_{\mu s_p}(y)$ tends to be close to $1$. 
Based on 
the sign of $f_{\mu}'(1; y) = 1-y+\mu$,
the proximal value is less than 1
if $0 \le y < 1+\mu$. 
The convexity of $f_{\mu}'(x; y)$ allows initializing the Newton algorithm with $x=1$.
Now consider letting $x_0=1+\epsilon$ for $\epsilon$ small. We then have that 
\begin{eqnarray*}
y-1-\epsilon & = & \mu(1+\epsilon)^{p-1} 
\amp \approx \amp \mu[1+(p-1)\epsilon].
\end{eqnarray*}
Rearranging this approximate equality yields
\begin{eqnarray*}
1+\epsilon & \approx & \frac{y+\mu(p-2)}{1+\mu(p-1)}.
\end{eqnarray*}
This quantity is greater than 1 if $y > 1+\mu$ and less than $y$ whenever 
$y > (p-2)/(p-1)$.

\subsubsection{Non-convex case}\label{sec:prox:nonconvex}
As claimed, the proximal map $\prox_{\mu s_p}$ can be set-valued if $p < 1$:
\begin{proposition}[\citet{marjanovic2012}, Theorem 1\footnote{
For a direct correspondence between \Cref{prop:prox1} and \citet[Theorem 1]{marjanovic2012}, $q \gets p$, $\beta \gets x$, $z \gets y$, $\lambda \gets \mu/p$, $h_a \gets r_p$, $\beta_a \gets \kappa_p$, and $\beta_* = z_p(y)$ when $\mu = 1$.
}]\label{prop:prox1}
	Let $z_p(y)$ be the implicit function defined as the root of equation $0=x - y + x^{p-1}$ greater than $m_p=(1-p)^{1/(2-p)}$, where $p\in(0,1)$ and $y\ge 0$. 
	Then the proximal map of $\mu s_p$, where $\mu \ge 0$, is given by
\begin{equation}\label{eqn:prox1}
	\prox_{\mu s_p}(y) = 
		\begin{cases}
			0, & \text{if~} 0 \le y < \mu^{1/(2-p)}r_p, \\
			\{0, y\kappa_p/r_p\},
			& \text{if~} y = \mu^{1/(2-p)}r_p, \\
			\mu^{1/(2-p)}z_p(\mu^{-1/(2-p)}y),
			& \text{if~} y > \mu^{1/(2-p)}r_p,
		\end{cases}
\end{equation}
where 
$\kappa_p = (2/p)^{1/(2-p)}m_p = [2(1-p)/p]^{1/(2-p)}$, and $r_p = \kappa_p + \kappa_p^{p-1}$.
Furthermore, $z(r_p) = \kappa_p$.
\end{proposition}
\begin{remark}\label{rem:proximal}
The following can be easily shown:
\begin{enumerate}
	\item $\lim_{p\downarrow 0} r_p = \infty$ and $\lim_{p\to 1} r_p = 1$;
	\item $\lim_{p\downarrow 0} \frac{r_p}{\sqrt{2/p}} =1$ and $\lim_{p\downarrow 0} \frac{\kappa_p}{\sqrt{2/p}} =1$;
	\item $\lim_{p\downarrow 0} \max_{y\ge r_p} [y-z_p(y)] = 0$, and $\lim_{p\to 1} \max_{y\ge r_p} [y-z_p(y)] = 1$;
	\item $\lim_{y\to\infty} \frac{y-z_p(y)}{1/y^{1-p}} = 1$. 
\end{enumerate}
Thus if $p\downarrow 0$, then $\prox_{\mu s_p}(y)$ tends to the hard thresholding operator
\begin{equation}\label{eqn:hard}
		\argmin_{x\in\Real}\left\{\frac{1}{2}(x-y)^2 + \frac{\mu}{p} |x|_0\right\}
		=
		\begin{cases}
			0, & |y| < \sqrt{2\mu/p}, \\
			\{0, \sqrt{2\mu/p} \}, & |y| = \sqrt{2\mu/p}, \\
			y, & |y| > \sqrt{2\mu/p},
		\end{cases}
\end{equation}
where $|x|_0 = 0$ if $x=0$ and $1$ otherwise.
When $p$ tends to $1$, it converges to the soft thresholding operator
	\[
		\argmin_{x\in\Real}\left\{\frac{1}{2}(x-y)^2 + \mu|x|\right\}
		=
		\begin{cases}
			0, & |y| \le \mu, \\
			\sgn(y)(|y| - \mu), & |y| > \mu.
		\end{cases}
	\]
\end{remark}

\Cref{prop:prox1} suggests that computing $z_p(y)$, or the root of the equation $0 = x - y + x^{p-1}$ that is greater than $m_p = (1-p)^{1/(2-p)}$, with high accuracy is a key to computing the set-valued map $\prox_{\mu s_p}(y)$ (for $y \ge 0$), when $p \in (0, 1)$.
This root is a potential minimizer of $f_1(x; y) = \frac{1}{2}(x - y)^2 + |x|^p/p$, which we abbreviate as $f(x)$ for simplicity, on $x \ge 0$. Its derivative $f'(x) = x - y + x^{p-1}$ is infinite at $x=0$ and convex on $x > 0$. The minimum of $f'(x)$ occurs at $m_p=(1-p)^{1/(2-p)}$ that is strictly positive.
If $f'(m_p)$ is non-negative (which means $y \le m_p + m_p^{p-1}$), then the minimum of $f(x)$ occurs at $x=0$.
Otherwise, $y \ge m_p + m_p^{p-1}$, and the minimum occurs to the right of $m_p$. Since $f'(x)$ is strictly increasing on $(m_p, \infty)$ and $f'(y) = y^{p-1}$, it must have a unique zero $z_p(y)$ in the interval $(m_p, y)$. 
It follows that $z_p(y)$ and $0$ contend for the minimum point of $f(x)$.
Determining which only requires comparing two quantities $f(z_p(y))$ and $f(0) = \frac{1}{2}y^2$. Hence the $r_p$ in \Cref{prop:prox1} needs not be computed.

\paragraph{Newton's method}
In computing $z_p(y)$, the Newton method \eqref{newtonprox2} can be employed without any modification.
This iteration necessarily converges to $z_p(y)$ from \emph{any} initial point in $(m_p, y)$ since $f'(x)$ is increasing and convex in this interval.
Since the convexity of $f'(x) = f_{1}'(x; y)$ remains intact with $p\in(0, 1)$, the choice of the initial point in the $p \ge 2$ case is still valid.
From \Cref{rem:proximal}, we see that if $\kappa_p/\sqrt{2/p}$ is very close to $1$ beyond the machine precision, it is safe to approximate the proximal map with hard thresholding \eqref{eqn:hard}.

\begin{remark}
	For $p=1/2$ and $p=2/3$, it can be shown that $z_p(y)$ has a closed form, so does $\prox_{\mu s_p}(y)$ \citep{xu2012,chartrand2016}. Our goal here is, however, to provide a unifying strategy of evaluating the dual function $g(\mu)$, for a wide range of $p$.
\end{remark}

\noindent \textit{Algorithm}  The discussion in this section is summarized in Algorithm \ref{alg:lpprox}.
\begin{algorithm}[t]\small
   \caption{Compute $\prox_{\mu s_p}(y)$ for $s_p(x)=|x|^p/p$,~$p\in(0,\infty)\setminus\{1, 2, \infty\}$}
   \label{alg:lpprox}
  \begin{multicols}{2}
  \begin{algorithmic}
   \STATE \textit{Input:} $y > 0$, $\mu\ge 0$, and $p>0$
   \IF{$p > 2$}
   		\STATE $x^{\star} \gets \text{NewtonRoot}(y, \mu, p)$
   \ENDIF
   \IF{$p > 1$}
		\IF {$\mu^{-1}(y/\mu)^{\frac{1}{1-p}}\approx 0$} \RETURN\ $(y/\mu)^{\frac{1}{p-1}}$ \ENDIF
   		\STATE $q \gets 1/(1 - 1/p)$
		\STATE $z \gets \text{NetwonProx}(y/\mu, 1/\mu, q)$
		\RETURN{$y - \mu z$}
   \ELSE
   		\STATE $x_{\min} \gets [(1 - p) \mu]^{1 / (2 - p)}$
		\IF{$f_{\mu}'(x_{\min}; y)<0$}
			\STATE $z \gets \text{NewtonRoot}(y, \mu, p)$
			\RETURN{$\argmin_{x\in\{0,z\}}f(x;y)$}
		\ELSE
			\RETURN{$0$} 
		\ENDIF
   \ENDIF
  \end{algorithmic}
  \columnbreak
  \begin{algorithmic}
   \STATE \textit{Subroutine} NewtonRoot:
   \IF{$\mu y^{p-1} \approx 0$}
   	\RETURN{$y$}
   \ENDIF
   \STATE $x \gets \max\{1, \min(y, [y + \mu(p-2)] / [1 + \mu(p-1)])\}$
   \REPEAT
	\IF{$x > 1$ or $p < 1$}
		\STATE $a \gets y / x^{p-2} + \mu(p-2)x$
      	\STATE $b \gets 1 / x^{p-2} + \mu(p-1)$
	\ELSE
      	\STATE $a \gets y + \mu(p-2)x^{p-1}$
      	\STATE $b \gets 1 + \mu(p-1)x^{p-2}$
	\ENDIF
	\STATE $x \gets a / b$
   \UNTIL{convergence}
   \RETURN{$x$}
  \end{algorithmic}
  \end{multicols} \vspace{-5pt}
\end{algorithm}
 \section{Maximizing dual objective via Newton ascent ($p > 1$)}\label{sec:convex}
From \Cref{sec:prox:convex}, we see that it suffices to consider the case $p > 2$ for convex norm balls.
Then
Slater's condition holds and strong duality implies that solving \eqref{eqn:dual} is equivalent to \eqref{eqn:projection}.
From \Cref{sec:prox:nonconvex} we see that function $z_p(y)$, implicitly defined as $0 = z_p(y) - y + [z_p(y)]^{p-1}$, is continuously differentiable in $y$, even for $p > 2$. Observe that in the latter case $\prox_{s_p}(y) = z_p(y)$.
Letting $x$ be non-negative without loss of generality, we may rewrite
\begin{eqnarray*}
f_\mu(x;y) & = & \mu^{2/(2-p)}\left[\frac{1}{2}\left(\frac{y}{\mu^{1/(2-p)}}
-\frac{x}{\mu^{1/(2-p)}}\right)^2+\frac{1}{p}\left(\frac{x}{\mu^{1/(2-p)}}\right)^p\right]
\\
 & = & \mu^{2/(2-p)}f_1(\tilde{x}; \tilde{y}),
 \qquad
 \text{where}
 \quad
 \tilde{x} = x/\mu^{1/(2-p)}, ~
 \tilde{y} = y/\mu^{1/(2-p)},
\end{eqnarray*} 
which asserts that
\begin{eqnarray}\label{eqn:proxscaling}
\prox_{\mu s_p}(y) & = & \mu^{1/(2-p)}\prox_{s_p}\big(y/\mu^{1/(2-p)}\big).
\end{eqnarray}
Thus
$x_i^{\star}(\mu)=\prox_{\mu s_p}(y_i) = \mu^{1/(2-p)}z_p(y_i/\mu^{1/(2-p)}) > 0$ is a continuously differentiable function of $\mu$, with derivative 
\begin{equation}\label{xderiv1}
{x_i^{\star}}'(\mu)  =  -\frac{{x_i^{\star}}(\mu)^{p-1}}{1+\mu(p-1){x_i^{\star}}(\mu)^{p-2}} 
\end{equation}
obtained by applying the implicit function differentiation rule to equation \eqref{stationarity_eq}.
It follows immediately that $x_i^{\star}(\mu)$ is strictly decreasing in $\mu$
and satisfies 
\[
	0<x_i^{\star}(\|\by\|_q) \le x_i^{\star}(\mu) \le x_i^{\star}(0) = y_i 
\]
on the interval $[0,\|\by\|_q]$. Expressions (\ref{stationarity_eq}) and (\ref{xderiv1}) also allow one to derive formulas for the derivatives of the dual objective $g$ in \eqref{eqn:dualobjective}
and verify that it is twice continuously differentiable:
\begin{align}
g'(\mu) 
&= \frac{1}{p}\sum_{i=1}^d {x_i^{\star}}(\mu)^{p} - \frac{1}{p}
= \frac{1}{p}\left(\|\bx^{\star}(\mu)\|_p^p - 1\right),
\label{eqn:gderiv1}
\\
g''(\mu) &= \sum_{i=1}^d {x_i^{\star}}(\mu)^{p-1}{x_i^{\star}}'(\mu) = - \sum_{i=1}^d \frac{{x_i^{\star}}(\mu)^{2p-2}}{1+\mu(p-1){x_i^{\star}}(\mu)^{p-2}} 
\label{eqn:gderiv2}
,
\end{align}
where \eqref{eqn:gderiv1} is due to Danskin's theorem; see, e.g., \citet[Proposition B.25]{Bertsekas99} or \citet[Proposition 3.2.6]{lange16}.
    
\begin{algorithm}[b]\small 
\caption{Dual Newton ascent}\label{alg:newton}
  \begin{algorithmic}
   \STATE \textit{Input:} $\by > 0$ with $\|\by\|_p>1$, $p>1$
   \STATE $q \gets p / (p - 1)$
   \STATE Choose $\mu \in (0, \|\by\|_q]$; $\alpha \in (0, 1/2)$; $\beta \in (0, 1)$
   \STATE \textit{Main loop:} 
   \REPEAT
   		\STATE $\Delta\mu \gets -g'(\mu)/g''(\mu)$
   		\STATE $t \gets 1$
		\WHILE[Armijo rule]{$g(\mu+t\Delta\mu) < g(\mu) + \alpha t g'(\mu)\Delta\mu$} 
			\STATE $t \gets \beta t$
		\ENDWHILE
		\STATE $\mu \gets \mu + t\Delta\mu$
   \UNTIL{convergence}
   \STATE $\bx^{\star} \gets (\prox_{\mu s_p}(y_1),\dotsc,\prox_{\mu s_p}(y_d))^T$
   \COMMENT{eq. \eqref{eqn:prox1}}
   \RETURN{$\bx^{\star}$}
  \end{algorithmic}
\end{algorithm}
 
Thus the dual problem \eqref{eqn:dual}
can be solved efficiently by a Newton method with backtracking 
\begin{eqnarray}\label{eq:newton_iterates}
\mu_{n+1} & = & \mu_{n} - t_n \frac{g'(\mu_n)}{g''(\mu_n)},
\end{eqnarray}
where the step size $t_n \in(0,1]$ 
ensuring the ascent property 
can be found by 
the Armijo rule.
Since 
$\lim_{\mu\to0^+}x_i^{\star}(\mu)=x_i^{\star}(0)=y_i$,
the directional derivative of $g$ at $0$ 
toward the positive direction
is 
$\frac{1}{p}(\|\by\|_p^p-1)  >  0$.
Also recall that the solution $\mu^{\star}$ to the dual \eqref{eqn:dual} lies in $(0, \|\by\|_q]$ as demonstrated in \Cref{sec:proximal:objective}.
Any initial point in this interval 
ensures all iterates remain in $(0,\|\by\|_q]$
by the ascent property of iteration \eqref{eq:newton_iterates}. 
The development so far is summarized in \Cref{alg:newton}. 

The following proposition establishes quadratic convergence of \Cref{alg:newton} and shows that backtracking is not needed after a few steps:
\begin{proposition}\label{prop:newtonrate}
The Newton iterates $\{\mu_k\}$
generated by \Cref{alg:newton} converge quadratically to the solution $\mu^\star$ to \eqref{eqn:dual} after a number of backtracks less than or equal to
	\begin{equation}\label{eqn:numbacktrack}
		\frac{M^2L^2/m^5}{\alpha\beta\min\{1, 9(1-2\alpha)^2\}}(g(\mu^{\star}) - g(\mu_0))
		,
	\end{equation}
where 
\begin{align*}
M &=\|\by\|_{2p-2}^{2p-2},
\quad
m = \sum_{i=1}^d\frac{[x_i^{\star}(\|\by\|_q)]^{2p-2}}{1+\|\by\|_q(p-1)y_i^{p-2}}, ~\text{and}
\\
L  &=(p-1)\sum_{i=1}^d [2M_i + \|\by\|_q N_i],  
~q = p / (p - 1), ~\text{with}
\\
& M_i =\max\big\{x_i^{\star}(\|\by\|_q)^{3p-4},\allowbreak y_i^{3p-4}\big\}, 
\quad 
N_i =\max\big\{x_i^{\star}(\|\by\|_q)^{3p-6}, y_i^{3p-6}\big\}
.
\end{align*}
\end{proposition} 
\begin{proof}
For a twice continuously differentiable and strongly convex objective function with a Lipschitz continuous Hessian, convergence of Newton's method is quadratic after a limited number of backtracks \citep[\S 9.5.3]{Boyd2004}. 
Thus it suffices to show that $g(\mu)$ in \eqref{eqn:dualobjective} satisfies these conditions.

We establish a negative upper bound on $g''(\mu)$ and a finite upper bound on the magnitude of the third derivative $|g'''(\mu)|$.
The second expression for $g''(\mu)$ in \eqref{eqn:gderiv2} and that $0<x_i^{\star}(\|\by\|_q) \le x_i^{\star}(\mu) \le y_i$  for $\mu\in[0,\|\by\|_q]$ make clear 
$$
\sum_{i=1}^d\frac{[x_i^{\star}(\|\by\|_q)]^{2p-2}}{1+\|\by\|_q(p-1)y_i^{p-2}}
= m
\le -g''(\mu)
\le M
= \|\by\|_{2p-2}^{2p-2}
$$
in this interval, since
$x_i^{\star}(\|\by\|_q) = \prox_{\|\by\|_q s_p}(y_i)$.

To see
that $g(\mu)$ is three times continuously differentiable, note that $x_i^{\star}(\mu)$ is twice continuously differentiable with second derivative
\begin{eqnarray*}
{x_i^{\star}}''(\mu) & = &
-\frac{(p-1)[{x_i^{\star}}(\mu)^{p-2}+\mu {x_i^{\star}}(\mu)^{2p-4}]}
{[1+\mu (p-1){x_i^{\star}}(\mu)^{p-2}]^2} \, {x_i^{\star}}'(\mu).
\end{eqnarray*}
Equation \eqref{xderiv1} 
further implies that 
${x_i^{\star}}''(\mu)$ is also bounded on the interval, so that
\begin{eqnarray*}
g'''(\mu) & = & \sum_{i=1}^d\Big[ (p-1){x_i^{\star}}(\mu)^{p-2} {x_i^{\star}}'(\mu)^2 + {x_i^{\star}}(\mu)^{p-1}{x_i^{\star}}''(\mu) \Big]
\end{eqnarray*}
is well-defined, and 
$|g'''(\mu)| \le (p-1)\sum_{i=1}^d [2M_i + \|\by\|_q N_i]=L$.

Expression \eqref{eqn:numbacktrack} follows from \citep[Eq. (9.40)]{Boyd2004}.
\qed
\end{proof}

\paragraph{Numerical consideration}
A simple numerical device greatly improves the stability of \Cref{alg:newton}.
In computing the Newton step in iteration \eqref{eq:newton_iterates}, overflow may occur when $|x^{\star}_i(\mu)|^p$ becomes too large. This phenomenon is prominent, of course, when $p$ is large. 
A remedy is to use $(|x^{\star}_i(\mu)| / \max_j |x^{\star}_j(\mu)|)^p$ to normalize the first \eqref{eqn:gderiv1} and second \eqref{eqn:gderiv2} derivatives of $g(\mu)$ when $(\max_j |x^{\star}_j(\mu)|)^p$ is too large, say, greater than $10^{10}$.

 \section{Maximizing dual objective via bisection}\label{sec:nonconvex}
Solutions when  $p\in(0, 1)$ are hindered by the lack of convexity of the unit ball. 
When $p>1$, bisection offers a slower alternative to the Newton methods for finding the root of $g'(\mu)$, the derivative of the dual objective function $g(\mu)$ in problem \eqref{eqn:dual}. The last expression in equation \eqref{eqn:gderiv1} shows that
bisection on $g'(\mu)$ is equivalent to that on the norm function $\|\bx^{\star}(\mu)\|_p$.
This function is continuously monotone decreasing, and we simply need to find the $\mu$ that corresponds to $\|\bx^{\star}(\mu)\|_p = 1$. 
For $p<1$, the delicacy lies in that solution sets $P_{B_p}(\by)$ and the associated $\bx^{\star}(\mu)$ can be multi-valued. Further, the dual function $g(\mu)$ is no longer smooth.

Fortunately, $g(\mu)$ is concave and single-valued even if $\bx^{\star}(\mu)$ is multi-valued, so that $g(\mu)$ always has well-defined directional derivatives. 
For concreteness, we define the radius
function 
\begin{equation}\label{eqn:radius}
r(\mu) = \max \{\|\bx^{\star}(\mu)\|_p : \bx^{\star}(\mu) \in \argmin_{\bx} \mathcal{L}(\bx,\mu)\}.
\end{equation}
It will be also convenient to define the maximum proximal operator of $\mu s_p$:
\[
	\max\prox_{\mu s_p}(y) = \max\{u: u \in \prox_{\mu s_p}(y)\}
	.
\]
Recall that $z_p(y)$ is the root of $0 = z - y + z^{p-1}$ greater than $m_p$. 
Since $z_p(y) > m_p$,
the function $\max \prox_{s_p}(y)$ is single-valued and right-continuous. It takes a jump of size $\kappa_p$ at $y=r_p$, and is strictly increasing from there. 
On the other hand,
from relation \eqref{eqn:proxscaling} and Proposition \ref{prop:prox1}, 
the map $\mu \mapsto \max \prox_{\mu s_p}(y)$ when $\mu > 0$ and $y > 0$ is held fixed is decreasing and left-continuous, with a discontinuity only at $\mu = (y/r_p)^{2-p}$, before which it is strictly decreasing on $(0, (y/r_p)^{2-p})$.

Since $\bx^{\star}(\mu) = (x_1^{\star}(\mu), \dotsc, x_d^{\star}(\mu))$ and $x_i^{\star}(\mu)=\prox_{\mu s_p}(y_i)$,
as $\mu$ tends to $\infty$ from $0$, $r(\mu)$ monotonically decreases from $\|\by\|_p$ down to $0$, and
is left-continuous.
Let $\bar{\mu}$ denote a discontinuity in $\bx^{\star}(\mu)$. Each $x_i^{\star}(\mu)$, and in turn $g(\mu)$, is single-valued and differentiable both at $\bar{\mu}-\epsilon$ and $\bar{\mu}+\epsilon$ 
for sufficiently small $\epsilon>0$.
Therefore, the subdifferential satisfies
\[
	\partial g(\bar{\mu}) = \left[\big(\lim_{\epsilon\downarrow 0}r(\bar{\mu}+\epsilon)-1\big)/p, \big(\lim_{\epsilon\downarrow 0}r(\bar{\mu}-\epsilon)-1\big)/p \right]
	,
\]
and we see that the inclusion $1 \in [\lim_{\epsilon\downarrow 0}r(\mu+\epsilon), \lim_{\epsilon\downarrow 0}r(\mu-\epsilon)]$ is necessary and sufficient for maximizing $g(\mu)$.
Bisection is therefore  guaranteed to return a sufficiently small interval $[\mu-\delta, \mu+\delta]$ such that $1 \in [r(\mu + \delta), r(\mu - \delta)]$ for any $\delta> 0$. 

As strong duality is no longer guaranteed, a final remaining concern is the possibility of nonzero duality gap. The following proposition shows that this possibility is rare.  
\begin{proposition}\label{prop:bisection}
Let $\omega = \inf \{\mu: r(\mu) \le 1\}$. 
Suppose $1 \in \nu(\omega)$ for the multi-valued norm function $\nu(\mu)=\{\|\bx^{\star}\|_p:
\bx^{\star} \in \argmin_{\bx} \mathcal{L}(\bx,\mu)\}$.
Then the point $\bx^{\star}(\omega) \in \argmin_{\bx}\mathcal{L}(\bx,\omega)$ with $\|\bx^{\star}(\omega)\|_p = 1$ 
solves \eqref{eqn:projection}.
\end{proposition}
\begin{proof} 
	Recall $\|\bx^{\star}(\omega)\|_p=1$.		
	For any $\bx$ satisfying $\|\bx\|_p \le 1$, we have 
\[
g(\omega) 
\amp = \amp
\mathcal{L}[\bx^{\star}(\omega),\omega] 
\amp = \amp
\frac{1}{2}\|\by-\bx^{\star}(\omega)\|_2^2 
\amp \le \amp \mathcal{L}(\bx,\omega) 
\amp \le \amp \frac{1}{2}\|\by-\bx\|_2^2
.
\]
Hence, $\bx^{\star}(\omega) \in P_{B_p}(\by)$
and the duality gap is zero.
\end{proof}

Bisection can thus be understood as finding the $\omega$ defined in this proposition.
If $r(\omega) = 1$, then the projection problem is solved. Assuming 
$r(\omega) < 1$ contradicts the left-continuity of $r(\mu)$ and the definition of $\omega$. 
Hence, $r(\omega)> 1$ and $r(\omega+\epsilon) < 1$.
Repeating the argument of Proposition \ref{prop:bisection}
yields the estimate
\begin{eqnarray*}
\frac{1}{2}\|\by-\bx^{\star}(\omega)\|_2^2 
+ \frac{r(\omega)^p - 1}{p}& \le & \frac{1}{2}\|\by-\bx\|_2^2 
\end{eqnarray*}
for all $\bx \in P_{B_p}(\by)$. Thus a small gap $r(\omega)-1$ implies a good approximation to a projected point---this is what we
find in practice.
\iffalse
\textcolor{red}{
In seeking a smaller gap,
suppose $r(\omega) > 1$ 
and without loss of generality the discontinuity is due to the multi-valuedness of $x_1(\omega)=\prox_{\omega s_p}(y_1)$.
Then by the discussion of the subsequent analysis, it must be
\[
	\omega = (y_1/r_p)^{2-p}
\]
and the size of the jump discontinuity is $\frac{y_1}{r_p}z(r_p)$, where the quantities $r_p$ and $z(r_p)$ are defined in the following section.
If $k$ components of $\by$ have the same value as $y_1$, then function $\nu(\mu)^p$ is $(k+1)$-valued at $\mu=\omega$ and the values are $r(\omega)^p - j \frac{y_1}{r_p}z(r_p)$, $j=0, 1, \dotsc, k$. 
Thus in this case the duality gap may be made smaller by choosing other points from $\argmin_{\bx}\mathcal{L}(\bx,\omega)$.
Our experiments based on random instances of $\by$ indicate this possibility is very rare.
}
\fi
Algorithm \ref{alg:bisection} provides pseudocode for the bisection method.
\begin{algorithm}[h]\small
\caption{Dual bisection}\label{alg:bisection}
  \begin{algorithmic}
   \STATE \textit{Input:} $\by > 0$ with $\|\by\|_p>1$, $p>0$
   \STATE $q^* \gets p / (p - 1)$ if $p \neq 1$, $q^* = \infty$ if $p = 1$
   \STATE $(\mu_l, \mu_r) \gets (0, \|\by\|_{q^*})$
   \REPEAT
   		\STATE $\mu_m \gets (\mu_l+\mu_r)/2$
		\IF[eq.  \eqref{eqn:radius}]{$(r(\mu_l)-1)(r(\mu_m)-1) < 0$} 
			\STATE $\mu_r \gets \mu_m$
		\ELSE
			\STATE $\mu_l \gets \mu_m$
		\ENDIF
   \UNTIL{convergence}
   \STATE $\bx^{\star} \gets (\prox_{\mu_m s_p}(y_1),\dotsc,\prox_{\mu_m s_p}(y_d))^T$
   \COMMENT{eq. \eqref{eqn:prox1}}
   \RETURN{$\bx^{\star}$}
  \end{algorithmic} \vspace{.2pt}
\end{algorithm}

 \section{Related work}\label{sec:review}

\subsection{Double bisection method for $p > 1$}
In \cite{liu2010efficient}, 
the $\ell_{1,q}$ proximal problem arising from $\ell_{1,p}$ multitask learning ($p>1$, $1/p+1/q=1$)
is studied:
\[ 
	\min_{\bZ \in \Real^{d\times s}} \frac{1}{2} \lVert \bZ - \bV \rVert_2^2 + r \sum_{i=1}^s \| \bz_i \|_q
,
\] 
where $\bz_i$ is the $i$th column of matrix $\bZ$.
Note that this splits into $s$ independent evaluations of the
proximal operator 
$\prox_{r\|\cdot\|_q}(\by)$
for solving \eqref{eqn:fenchel}.
If $\by>\mathbf{0}$ and $r=1$,
the optimal solution $\bz^\ast$ to \eqref{eqn:fenchel} satisfies the stationary condition
\begin{equation}\label{eqn:fenchelstationary} 
	\bz^\ast - \by +  \| \bz^\ast \|_q^{1-q} \bz^{\ast (q-1)} = \mathbf{0} , 
\end{equation}
where $\bu=\bx^{(q-1)}$ is defined elementwise by $u_i=\sgn(x_i)|x_i|^{q-1}$.

The value of the optimal multiplier $c^\ast = \| \bz^\ast \|_q^{1-q}$ is unknown, so it is proposed to determine $c^\ast$ by finding the root to the auxiliary function
\begin{equation}\label{eqn:aux1} 
	\phi(c) = \psi(c) - c , \quad c \geq 0, 
\end{equation}
where $\psi(c) = \Big[ \sum_{i=1}^n ( \omega_i^{-1} (c) )^q \Big]^{\frac{1-q}{q}}$ with $\omega_i(z) = \frac{y_i - z}{z^{q-1}}$,  $0 < z \leq y_i$.
It is shown that $\phi(c)$ 
has a unique root in 
$[\min_{i=1,\dotsc,d}\gamma_i, \max_{i=1,\dotsc,d}\gamma_i]$ where
\begin{equation}\label{eqn:bisectintv}
	\gamma_i = \frac{1-\epsilon}{\epsilon^{q-1} y_i^{q-2}},
	\qquad
	\epsilon = \frac{\|\by\|_p - 1}{\|\by\|_p} \in (0, 1).
\end{equation}
Hence its root is found by bisection. Evaluation of $\phi(c)$ requires inversion of the third auxiliary function $\omega_i(z)$ for each $i$, which is accomplished by finding the root of yet another auxiliary, monotone function 
\begin{equation}\label{eqn:aux4}
    h_c^{y_i}(z) = z - y + c z^{q-1} 
\end{equation}
by bisection.
This constitutes a nested or double bisection algorithm.

For $p>1$, it is now clear that the root of the auxiliary function \eqref{eqn:aux4} of \cite{liu2010efficient}
is equivalent to equation \eqref{stationarity_eq} with substitution $p\gets q$ and $\mu\gets c$, i.e., $\prox_{cs_q}(y_i)$. Thus we have also shown that at least one of the nested bisection routines for solving \eqref{eqn:fenchel} \cite{liu2010efficient} can be replaced by the faster Newton method (Algorithm \ref{alg:lpprox}). 
The other root finding of the more esoteric auxiliary function $\phi(c)$ is related 
to solving
\[
    \|\by - c^2\bx^{\star}(c^{1-p})\|_q^{1-q} = c
\]
via
Moreau's identity, $$\prox_{cs_q}(y_i) = y_i - c\prox_{c^{-1}s_p}(y_i/c) = y_i - c^2 x_i^{\star}(c^{1-p}).$$
In solving this equation via bisection, the interval \eqref{eqn:bisectintv} involves both $1/y_i^{q-2}$ and $(\|\by\|_p-1)/\|\by\|_p$, so when $p$ is large ($q$ is close to 1), initial values may run into numerical difficulties.
In contrast, our more transparent approach  solves the dual optimality condition $g'(\mu)=0$ for \eqref{eqn:dual}, or equation
\begin{equation}\label{eqn:rootfinding}
    \|\bx^{\star}(\mu)\|_p^p = 1 
    ,
\end{equation}
in the straightforward interval $[0, \|\by\|_p]$.
The left hand side is twice continuously differentiable with respect to $\mu$, and so Newton's method (Algorithm \ref{alg:newton}) efficiently solves this equation.
So when $p>1$, our approach constitutes a ``double Newton'' method. If $p<1$, the outer Newton is replaced with bisection; the inner Newton 
for $\prox_{s_p}(\cdot)$ remains intact.

\subsection{Projected Newton for $p > 1$}

Assuming again $\by > \mathbf{0}$ and $r=1$ without loss of generality, \cite{barbero2018} proposes to solve \eqref{eqn:fenchel} together with a redundant constraint $\bz \ge \mathbf{0}$
using projected Newton \citep{bertsekas1982projected}.
Let $\tilde{g}$ denote the objective of \eqref{eqn:fenchel}.
An inactive component of $\bz$ is defined as the $i$th component with either $z_i>0$ or $z_i=0$ but $\nabla_i \tilde{g}(\bz)<0$.
If we denote the set of these by $I$, 
$\tilde{g}$ is differentiable within $I$: 
\begin{equation}\label{eqn:projnewt0}
	\begin{split}
    \nabla_I \tilde{g}(\bz) &= \bz_I - \by_I + \bar{\bz}_I, 
	\qquad
	\bar{\bz}=(\bz/\|\bz\|_q)^{q-1},
    \\
    \nabla^2_I \tilde{g}(\bz) &= \text{diag}(\bv_I) + c\bar{\bz}_I\bar{\bz}_I^T,
	~~
	\bv=\mathbf{1} - c(\bz/\|\bz\|_q)^{q-2}, 
	~
	c=(1-p)\|\bz\|_q^{-1}
    ,
	\end{split}
\end{equation}
where $\nabla_I$ and $\nabla^2_I$ denote the gradient and Hessian restricted within $I$, respectively, and
$\bv_I$ refers to the subvector of $\bv$ indexed by $I$.
The Newton update of inactive components $\bz_I$ is then projected onto the nonnegative orthant:
\begin{equation}\label{eqn:projnewt}
    \bz_I := \left[ \bz_I - t\left(
    \bv_I^{-1}\odot\nabla_I g(\bz)
    - \frac{(\bv_I^{-1}\odot\bar{\bz}_I)(\bv_I^{-1}\odot\bar{\bz}_I)^T\nabla_I g(\bz)}
    {1/c + \bar{\bz}_I^T(\bv_I^{-1}\odot\bar{\bz}_I)}
    \right)
    \right]_{+}
    ,
\end{equation}
where $(\cdot)^{-1}$ and $\odot$ denote elementwise inverse and multiplication, respectively; $[\bx]_+=\max(\bx,\mathbf{0})$ componentwise.
The step size $t$ is chosen by using a backtracking line search. 

\iffalse
One can see that although the Hessian always remains positive semidefinite because the objective is convex, the curvature near the boundary ($\bz=\mathbf{0}$) can become extreme. This suggests numerical instability and possible overflows; indeed we encounter problems in our experiments for $p \in (1,2)$.
\fi

Per-iteration complexity of the projected Newton method \eqref{eqn:projnewt} is lower than the univariate dual method \eqref{eq:newton_iterates}, since the latter requires Algorithm \ref{alg:lpprox} internally to evaluate $x_i^{\star}(\mu)$. 
However, the construction of the restricted Hessian \eqref{eqn:projnewt0} indicates that the curvature of the objective of \eqref{eqn:fenchel} near the boundary ($\bz=\mathbf{0}$) can become extreme. 
This suggests numerical instability and possible overflows. 
The reference implementation\footnote{Available at \url{https://github.com/albarji/proxTV/blob/master/src/LPopt.cpp}} by the authors of \cite{barbero2018} faces this problem by adding several \textit{ad hoc} safeguards, including a switch to gradient descent. 
Despite this, we encountered numerical inaccuracies using projected Newton in our experiments under large values of $p$; see \Cref{sec:performance}.

\subsection{Bisection method of Chen et al. \citep{chen2019outlier} for $p < 1$}

A referee pointed out potential similarities between the work of Chen, Jiang, and Liu \citep{chen2019outlier}  and our univariate dual method. 
As a subproblem of a low-rank matrix decomposition problem,
\citet[Sect. IV-A]{chen2019outlier}  considers problem \eqref{eqn:projection} and arrives at equation \eqref{eqn:rootfinding} via the Karush–Kuhn–Tucker (KKT) conditions derived from the Lagrangian of the reformulated primal \eqref{eqn:notsimple}. 
This work also suggests finding the root of equation \eqref{eqn:rootfinding} by bisection. Regarding the associated proximal map $\prox_{\mu s_p}(y_i)$ for $p < 1$, \citet{chen2019outlier} considers equation \eqref{stationarity_eq} directly from the KKT conditions. Since finding the root of equation \eqref{stationarity_eq} alone is not sufficient for fully evaluating $\prox_{\mu s_p}(y_i)$, an additional heuristic is developed \citet[Theorems 3 and 4]{chen2019outlier}.

While close to our approach, the work of Chen et al. ignores that solving \eqref{eqn:rootfinding} is in fact equivalent to solving the dual \eqref{eqn:dual}.
As a result, \citet{chen2019outlier} fails to capitalize that an (outer) Newton method can be employed to yield much faster convergence for $p > 1$ (see \Cref{prop:newtonrate}), while on the other hand when $p<1$, misses 
discontinuity of the target function $g'(\mu)$ as well as the possibility of nonzero duality gap.
Our inspection of the dual problem \eqref{eqn:dual} also 
brings focus to
the map $\prox_{\mu s_p}(\cdot)$, 
which is well-studied for $p < 1$ and more principled than solely analyzing equation \eqref{stationarity_eq}.
In fact, Theorems 3 and 4 of \citet{chen2019outlier} are subsumed by \Cref{prop:prox1} due to \citet{marjanovic2012}, which predates \citet{chen2019outlier} and our present paper by several years.

\subsection{MM algorithms for $p < 1$}
While this letter is under review, 
an iterative re-weighted $\ell_1$-ball projection (IRBP) algorithm
has been posted online as a preprint \citep{yang2021towards}. 
The main idea behind the IRBP algorithm is to ``smooth'' the nonconvex unit $\ell_p$ norm ball $B_p=\{\bx=(x_1, \dotsc, x_d): \sum_{i=1}^d|x_i|^p \le 1\}$ by
$$
    B_{p,\bepsilon} = \{\bx: \sum_{i=1}^d|x_i + \epsilon_i|^p \le 1\}
$$
for $\bepsilon = (\epsilon_1, \dotsc, \epsilon_d) > \mathbf{0}$
(recall that we set $r=1$),
and iteratively relax $B_{p,\bepsilon}$ by a weighted $\ell_1$ norm ball
$$
    r_n B_{1, w_n} = \{\bx: \sum_{i=1}^d w_{n,i}|x_i| \le r_n\}
$$
for the $n$-th iterate $\bx_n=(x_{n,1},\dotsc,x_{n,d})$. Here
$r_n = 1 - \sum_{i=1}^d(x_{n,i} + \epsilon_i)^p + \sum_{i=1}^d w_{n,i}|x_{n,i}|$
and
$w_{n,i} = p(|x_{n,i}|+\epsilon_i)^{p-1}$.
The next iterate is obtained by minimizing $f_0(\bx)=\frac{1}{2}\|\bx-\by\|_2^2$ on $r_n B_{1, w_n}$, which can be efficiently solved by a trivial modification of (unweighted) $\ell_1$-ball projection algorithms, e.g., \citep{duchi2008,condat2016fast}. 
The authors of \citet{yang2021towards} show that, for a certain dynamic update strategy for $\bepsilon$, every cluster point of the iterate $\{\bx_n\}$ is a stationary point of problem \eqref{eqn:projection}.

We note that IRBP is an instance of majorization-minimization (MM) algorithms \citep{lange16}. With the smoothed $\ell_p$ norm ball we aim to minimize $f_{\bepsilon}(\bx) \triangleq f_0(\bx) + \iota_{B_{p,\bepsilon}}(\bx)$, which is a perturbed objective for the unconstrained problem \eqref{eqn:simple} that is equivalent to the original problem \eqref{eqn:projection}. 
For each iterate $\bx_n$, the surrogate function
$$
    g(\bx|\bx_n) = f_0(\bx) + \iota_{r_nB_{1,w_n}}(\bx)
$$
majorizes $f_{\bepsilon}(\bx)$ at $\bx_n$, i.e., $g(\bx|\bx_n) \ge f_{\bepsilon}(\bx)$ for all $\bx$ and $g(\bx_n|\bx_n) = f_{\bepsilon}(\bx_n)$, since $r_nB_{1,w_n} \subset B_{p,\bepsilon}$ and $\iota_{B_{p,\bepsilon}}(\bx_n) = \iota_{r_nB_{1,w_n}}(\bx_n) = 0$.
By iteratively minimizing the surrogate function and driving $\bepsilon \downarrow \mathbf{0}$, the unperturbed problem \eqref{eqn:simple} is expected to be solved.

IRBP
is a primal algorithm as opposed to our dual bisection algorithm. 
With a proper scheduling for driving the $\bepsilon$ down to zero, IRBP converges to a feasible stationary point of the primal \eqref{eqn:projection} from any feasible initial point.
Optimality of the convergent stationary point depends on the choice of the initial point.
On the other hand,
since the dual objective $g(\mu)$ is concave, the dual bisection method can find 
the dual optimum 
from any initial point. 
If the duality gap is zero, or equivalently the root of \eqref{eqn:rootfinding} exists, then the 
the primal optimum
is found (\Cref{prop:bisection}). 
However,
if the duality gap is positive or $g'(\mu) - 1$ has a sign-changing discontinuity, then 
the primal solution recovered from the dual optimum may not even be feasible,
contrary to IRBP.
Although we have found that this possibility is rare in practice (see the next section), examples exhibiting nonzero duality gap do exist (see Example 5.1 of \citet{yang2021towards}).
In this case, rescaling the primal solution by its $\ell_p$ norm results in a feasible point.
Although there is no guarantee that the stationary conditions are met,
our experience tells that this rescaling often yields satisfactorily small primal objective values.
 \section{Empirical results}\label{sec:performance}
\subsection{Multi-task learning}
Our empirical assessment of the proposed methods begins with an application to multi-task learning under $\ell_{1,p}$ regularization. 
Let $\bA \in \Real^{m \times d}$ be a design matrix containing data with feature dimension $d$, and $\bY \in  \Real^{m \times k}$ be the matrix of response variables, where the columns are observations corresponding to $k$ tasks. We seek the matrix $\bB \in \Real^{d \times k}$ with rows denoted  $\bB_{(i,\cdot)}$ as the solution to 
\begin{equation}
 \argmin_{\bB \in \Real^{d \times k} } \frac{1}{2} \lVert \bA \bB - \bY \rVert_F^2  + \tau \sum_{i=1}^d \lVert \bB_{(i,\cdot)} \rVert _p .
\label{eq:multitask}
\end{equation}
The second term promotes row-wise sparsity by way of an $\ell_{1,p}$ norm, and coincides with the usual group lasso when $p=2$.   \citet{zhang2010} motivates other choices of $p$ implying different ``group discounts" to the loss, showing that proper choice of $p \in (1, \infty] \setminus \{2, \infty\}$ can significantly improve performance.  \citet{liu2010efficient} confirms this finding and develops a more efficient double-bisection algorithm mentioned above. This is used to evaluate proximal maps for $\ell_{1,p}$ norms in the popular package SLEP \cite{liu2011slep}, and more recently by \citet{sra2012}, \citet{vogt2012}, and \citet{zhou2015ell1p} within a proximal gradient algorithm for fitting \eqref{eq:multitask}. 

Following the data generation in \cite{zhou2015ell1p}, we show that replacing double-bisection by our dual Newton ascent (Algorithm 2) reduces runtime by orders of magnitude. 
We draw entries of the covariate matrix $\bA$ as standard Gaussians shifted to have mean $2$. 
We choose $10$ rows (groups) of the true $d$-by-$k$ coefficient matrix   $\bB^\ast$ to be nonzero, drawn as standard Gaussian vectors. 
Then $\bY = \bA \bB^\ast + \bZ$ where entries of $\bZ$ are independent zero mean Gaussian with standard deviation $0.1$. 
The projection tolerances for all instances in the multi-task learning example are set to $10^{-6}$, with iteration limit $5000$ per projection step. We note that for $p \approx 5$ and above, the double bisection approach reaches the maximum iteration limit in many instances. We focus on efficiency as all methods run within the same outer proximal gradient algorithm, and reach identical solutions up to specified relative tolerance criterion of $10^{-3}$. Due to runtime considerations of the competing method, we fix $\tau$ at a constant that scales with the product of $d \times p$ rather than choose by cross-validation for each trial. Our algorithm and the competing method are initialized at identical starting value, obtained by adding another standard Gaussian to the true solution. The proximal gradient step size is set to $1/2$ divided by the largest eigenvalue of $\bA^T\bA$.

Fig. \ref{fig:timing} shows that as curvature increases with $p$, the method of \cite{liu2010efficient} struggles even at small scales, while our method consistently terminates in a fraction of a second. Methods are run under matched relative tolerance and reach identical solutions at convergence. 
The right panel shows that our algorithm remains tractable when the dimension times the number of tasks reaches millions, converging in several minutes consistently over a wide range of choices $p$.

 \subsection{High-dimensional projections}
Inspired by their success in the context of multi-task learning, we now examine the runtime, accuracy, and scalability of the proposed methods more closely.
We consider projecting onto $\ell_p$ balls in dimension $d=1,000,000$.
Having already established the limitations of the double-bisection method, we compare our dual Newton ascent and bisection methods to the projected Newton \citep{barbero2018} for $p > 1$,
and to the IRBP method \citep[Algorithm 1]{yang2021towards} for $p < 1$.
Since the IRBP method involves a crucial subproblem of projection onto a weighted $\ell_1$ ball, we employ three implementations of weighted $\ell_1$-ball projection:
	\begin{itemize}
	\item IRBP1: weighted version of Condat's algorithm \citep{condat2016fast}, in which catastrophic cancellation is avoided at the expense of computational complexity;
	\item IRBP2: weighted version of Duchi et al's algorithm \citep{duchi2008};
	\item IRBP3: reference implementation by the authors of \cite{yang2021towards},\footnote{Available at \url{https://github.com/Optimizater/Lp-ball-Projection}} which employs an alternating projection method.
	\end{itemize}

Additionally, a na{\"i}ve method making use of the nearest available exact projection, choosing the $\ell_\infty$ ball for $p \in (4,\infty)$, the $\ell_2$ ball for $p \in (\frac{3}{2},4]$, the $\ell_1$ ball for $p \in (\frac{1}{2},\frac{3}{2}]$, and the $\ell_0$ ball for $p \in (0,\frac{1}{2}]$ is compared; these nearest exact projections are then scaled to observe the $\ell_p$-ball constraint.

Convergence of the algorithms are declared as follows.
For the dual Newton ascent (Algorithm 2), convergence is declared 
when the distance $\text{obj}_n = \|\by-\bx_n\|_2$ of the current iterate $\bx_n$ to $\by$ satisfied the inequality 
\begin{eqnarray*}
|\text{obj}_{n} - \text{obj}_{n-1}| & < & 10^{-12}(1+\text{obj}_{n-1}).
\end{eqnarray*}
For bisection, Algorithm 3 is run until $\mu_{r}-\mu_{l} < 10^{-12}$ and either of the following criteria is met:
\[
	r(\mu_{l}) - r(\mu_{r}) < 10^{-7}(1 + r(\mu_{l}))
	\quad\text{or}\quad
	\mu_{r} - \mu_{l} < \epsilon_{\text{mach}} \mu_{r}
	,
\]
where $\epsilon_{\text{mach}}\approx 2.22\times 10^{-16}$ is the machine epsilon. These criteria are needed to cope with the discontinuity of the function $r(\mu)$.
The convergence criteria for the projected Newton \citep{barbero2018} 
and IRBP \cite{yang2021towards} follow the reference implementations, whose URLs are provided in the footnotes at the end of \Cref{sec:convex} and in this subsection. 

All the simulations were run on a Linux machine with an Intel Xeon E5-2650 v4 CPU @ 2.20GHz with 12 cores. A single core was used for each value of the power $p$. The code was written in the Julia programming language, except the projected Newton for which the complied C++ reference implementation was directly called from Julia, 
and IRBP3 for which the reference implementation in Python was  called via \texttt{PyCall.jl}.\footnote{Available at \url{https://github.com/JuliaPy/PyCall.jl}.} 

Results under several performance measures are reported in 
\Cref{table_cvx,table_noncvx}.
The components of each exterior point to be projected, $\by$, were sampled as independent standard normal entries. The radius $r$ of the $\ell_p$ ball for a given $\by$ was chosen uniformly from $(0,\|\by\|_p)$. All performance measures for a given $p$ and method represent averages over $100$ independent trials. The range of powers $p$ considered are designed to elicit both typical and extreme behavior. 
Runtime is assessed via number of iterations as well as elapsed time in seconds until convergence. 
Since the radius $r$ varies widely across the sampled external points, all the performance measures except runtime were computed after normalizing the coordinates, i.e., $x_i^{\star} \gets x_i^{\star}/r$, $y_i \gets y_i/r$, $i=1,\dotsc, d$. 

The objective value 
(``obj'') at convergence
must be considered together with the KKT measures, defined as follows. 
The ``KKT1'' measure is the sum of absolute values of the right-hand side of equation \eqref{stationarity_eq} for $i=1, \dotsc, d$: 
\begin{equation}\label{eqn:computedmu}
	\text{KKT1} = \sum_{i=1}^d \left|x^{\star}_i - y_i + \mu^{\star}|x^{\star}_i|^{p - 1}\sgn(y_i) \right|
	.
\end{equation}
The univariate dual methods (\Cref{alg:newton,alg:bisection}) compute the dual optimal variable $\mu^{\star}$; this measure can be directly calculated. However, other methods (na{\"i}ve and projected Newton) do not generate this dual variable, and  hence $\mu^{\star}$ is estimated by the formula
\[
	\mu^{\star} = (\by^T\bx^{\star} - \|\bx^{\star}\|_2^2)/\Vert\bx^{\star}\Vert_p^p
\]
(see equation \eqref{eqn:stationarycond} in \Cref{sec:proximal:objective}).
This computed dual variable may be negative, but we nevertheless computed the KKT1 measure. 
If $|x_i^{\star}|$ 
is very small (we used the threshold of $10^{-12}$) but $y_i$ 
is not, then $|x_i^{\star}|^{p-1}$ 
is approximated by $|(\mu^{\star})^{-1}y_i|$ 
(see equation \eqref{eqn:safe} in \Cref{sec:prox:convex});
this scenario is  encountered usually when $p$ is close to one.
This correction is also valid for the non-convex case ($0 < p <1$), in which equation \eqref{stationarity_eq} may not hold for every coordinate (see \Cref{rem:proximal} and \Cref{sec:prox:nonconvex}).

It may be argued that KKT1 is favorable to the univariate methods, since the projected Newton solves a different dual problem \eqref{eqn:fenchel}. For this reason, another KKT measure (``KKT2'') quantifies the deviation from the optimality condition \eqref{eqn:fenchelstationary} of problem \eqref{eqn:fenchel}: 
\[
	\text{KKT2} = \sum_{i=1}^d\left| z^{\star}_i - y_i + (|z^{\star}_i| / \|\bz^{\star}\|_q)^{q - 1}\sgn(y_i) \right|
	,
	\quad
	q = \frac{p}{p-1},
\]
which is valid only for $p \ge 1$,
where $\bz^{\star} = \by - \bx^{\star}$ is the computed optimal variable of problem \eqref{eqn:fenchel}.
A similar numerical caution is warranted if $|z_i^{\star}|$ 
is infinitesimally small, but this case usually occurs when $q$ is close to one, or $p$ is large. 
As a related measure, the ratio $\frac{1}{r}\|\bx\|_p$ informs whether the projected point $\bx$ falls on the surface of the $rB_p$; its difference from $1$ measures the duality gap provided that the KKT measures are small (see the discussion below \Cref{prop:bisection}).
All measures greater than $10^{10}$ are marked by $\infty$.

For $p < 1$, instead of KKT2 we measure the duality gap 
$$
    f_0(\bx^{\star}) - g(\mu^{\star})
    .
$$
If the  $\mu^{\star}$ in \eqref{eqn:computedmu} is negative, then $x_i(\mu^{\star})=\prox_{\mu^{\star}s_p}(y_i)$ is undefined, hence we set $g(\mu^{\star})=\texttt{NaN}$ and count the number of trials yielding \texttt{NaN}s.
Note that, if $\bx^{\star}$ is infeasible, this metric may be misleading.

\begin{figure}
\centering
\includegraphics[width=0.45\textwidth]{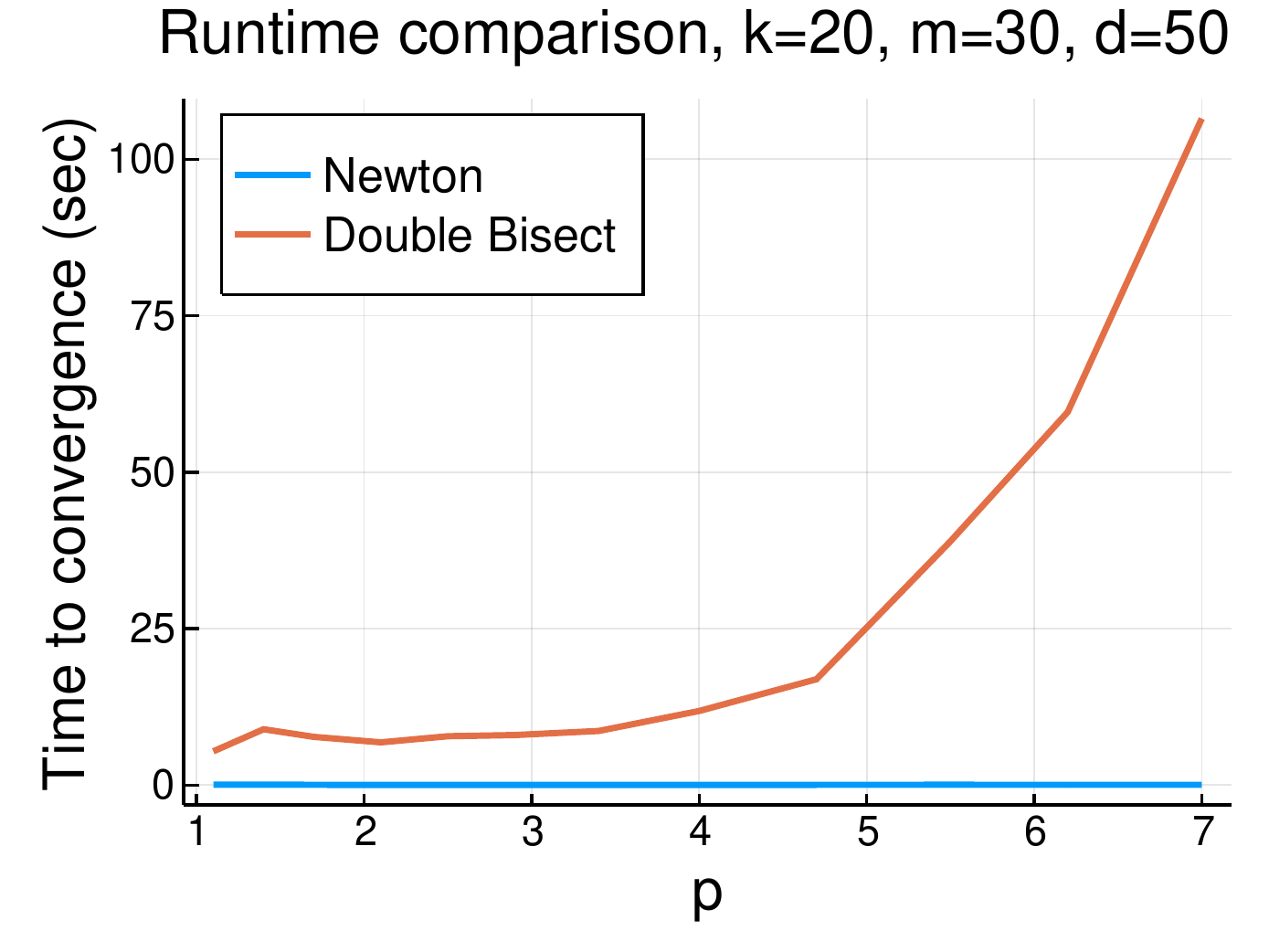} \includegraphics[width=0.45\textwidth]{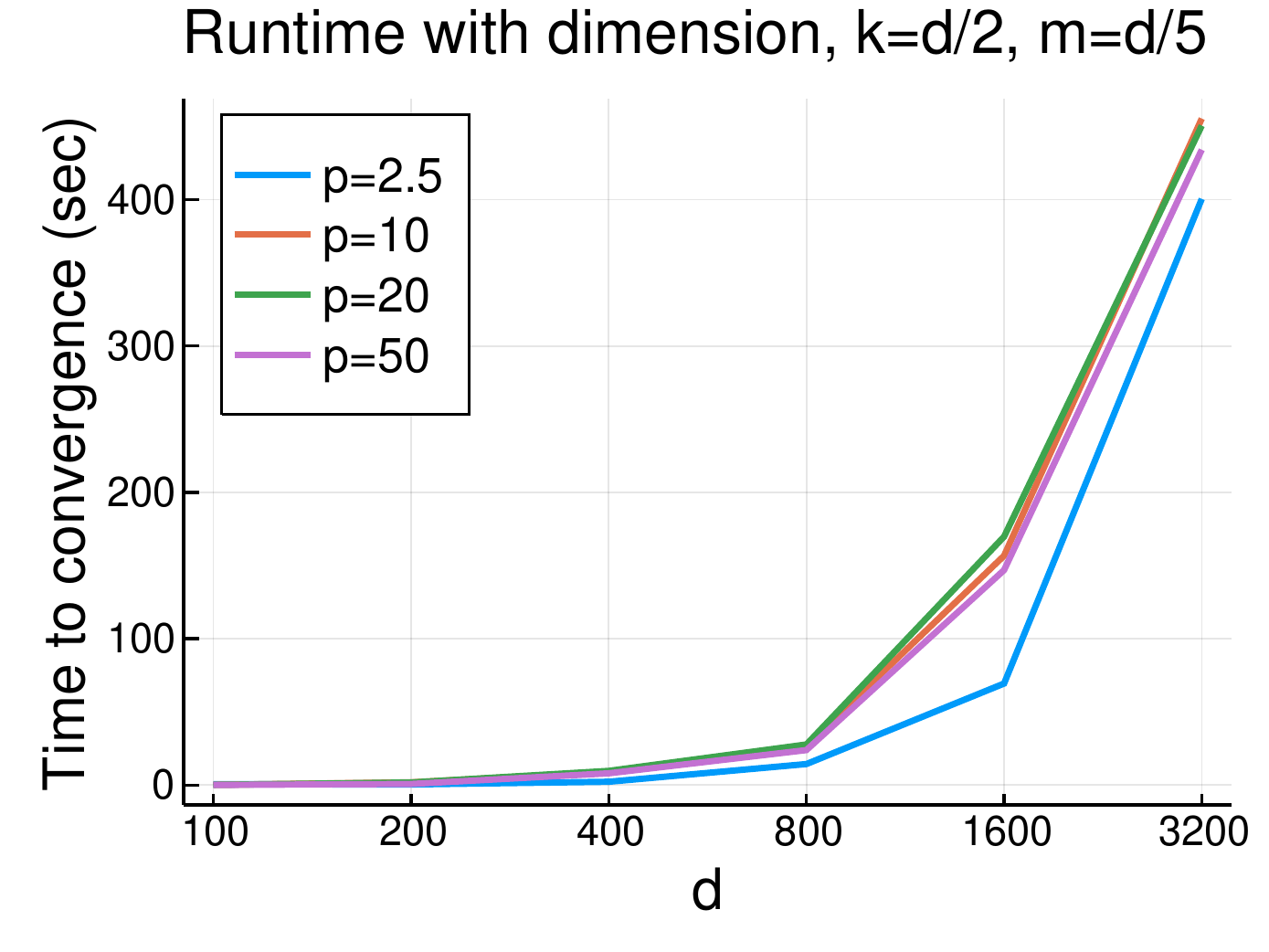}
\caption{Prior method used in SLEP slows as $p$ grows even on moderate problem size; our method scales to $d,k$ in the thousands consistently over a wide range of $p$.}
\label{fig:timing}
\end{figure}

\iffalse
The definitions of the performance measures presented in Table \ref{table1} deserve mention. 
Second, 
the recorded time (secs) for a given method and value of $p$ is elapsed
time until convergence. These times are less important than their ratios. 
Third, the objective value $\text{obj}$ represents the Euclidean distance of the converged point $\bx$ to $\by$, which must be considered together with the KKT measure defined as the $\ell_1$ norm of the components on the right side of equation  (\ref{stationarity_eq} (that is, a solution with lower objective but high KKT measure does not respect the constraint, i.e. is outside of the ball onto we wish to project). Related, the ratio $\frac{1}{r}\|\bx\|_p$ informs whether the projected point $\bx$ falls on the surface of the $rB_p$. All measures greater than $10^{10}$ is marked by $\infty$.

All measures greater than $10^{10}$ is marked by $\infty$.
\fi

Our results in \Cref{table_cvx} indicate that over a wide range of $p>1$, the dual Newton method successfully computes the projections with accuracy comparable to that of bisection in a fraction of the runtime. 
The projected Newton is faster, which may be expected due to its C++ implementation as opposed to Julia. However, both KKT measures and constraint violation are at least an order of magnitude greater than dual Newton and bisection given similar numbers of iterations to converge, even though the KKT2 measure is favorable to this method by construction. In particular, the accuracy across all measures becomes noticeably worse as $p$ increases, which is anticipated from the discussion in \Cref{sec:convex}.
For $p=1.01$ and $p=100$, the reference implementation of the projected Newton rounds them to $p=1$ and $p=\infty$, respectively.
The resulting accuracy, along with those of
the na{\"i}ve method, 
serves to illustrate the inadequacy of working only with computationally convenient projection operators.

The results for $p \in (0,1)$ are presented in \Cref{table_noncvx}. Recall that in this non-convex setting, only bisection and IRBP are meaningful options. No rescaling for observing the $\ell_p$-ball constraint is employed for either method. Nevertheless, the dual bisection method consistently delivered accurate projections for all the values of $p$ tried, indeed with small duality gaps, supporting the theoretical finding in Proposition \ref{prop:bisection}.
The behavior of IRBP is a bit complicated.
For $p \geq 0.5$ all three versions worked well with outcomes comparable to the dual bisection. (IRBP3 was excluded for $p=0.5$  because it took more than two hours for each trial: 10134, 8358, 17031, 28022 seconds in the first four trials.) IRBP converged in fewer iterations than bisection when $p$ is greater than $0.5$, but constraint violation is at least an order of magnitude greater than bisection as well as the objective values;
this is likely due to the convergence criteria, which followed the reference implementation (IRBP3). So, for this range of $p$, it appears that the two methods are comparable.
For $p < 0.5$, however, both IRBP1 and IRBP3 tended to drive the iterates toward zero, while IRBP2 produced outputs that were the same as the inputs. For neither results we could not call for accuracy.
Finally, it is interesting to note the performance of the na{\"i}ve method when $p$ is less than $0.5$. The objective value was within 0.001\% of bisection, while respecting the norm ball constraint. Not surprisingly, na{\"i}ve solutions for this range of $p$ were not dual feasible.
 
\begin{table}[tbh]
\caption{Average performance of $\ell_p$-ball projection algorithms for $d=1,000,000$ ($p > 1$)}\label{table_cvx} 
\begin{center}
\tiny
\begin{tabular}{lccccccc}
Method & $p$ & Iters & Secs & KKT1 & KKT2 & Obj & $\frac{1}{r}\|\mathbf{x}^{\star}\|_p - 1$ \\
\hline
Naive & 1.01 & 1.000 & 0.1963 & 0.1390 & 2.053 & 503.593 & -2.769e-15 \\
Dual Newton & 1.01 & 4.200 & 12.48 & 3.020e-9 & 1.663e-8 & 494.572 & 1.466e-8 \\
Bisection & 1.01 & 26.83 & 40.23 & 3.020e-9 & 3.725e-8 & 494.572 & 8.673e-10 \\
Projected Newton & 1.01 & 0.000 & 0.4648 & 0.1927 & 0.6943 & 544.706 & -0.1164 \\
\\
Naive & 1.05 & 1.000 & 0.1812 & 0.8063 & 2.946 & 563.572 & 4.741e-16 \\
Dual Newton & 1.05 & 4.120 & 11.99 & 1.187e-8 & 1.764e-8 & 447.537 & 9.759e-9 \\
Bisection & 1.05 & 26.66 & 39.28 & 1.188e-8 & 4.34e-8 & 447.537 & -1.548e-9 \\
Projected Newton & 1.05 & 6.180 & 2.656 & 1.637e-6 & 1.305e-6 & 447.537 & 1.931e-8 \\
\\
Naive & 1.1 & 1.000 & 0.1725 & 1.898 & 4.405 & 745.935 & 1.887e-17 \\
Dual Newton & 1.1 & 4.090 & 10.48 & 1.138e-7 & 1.617e-7 & 462.819 & 6.613e-8 \\
Bisection & 1.1 & 27.78 & 36.28 & 1.134e-7 & 2.962e-8 & 462.819 & -1.961e-9 \\
Projected Newton & 1.1 & 5.410 & 2.366 & 9.094e-5 & 1.053e-6 & 462.819 & 2.115e-8 \\
\\
Naive & 1.5 & 1.000 & 0.1614 & 10.42 & 38.50 & 1291.36 & -1.588e-16 \\
Dual Newton & 1.5 & 4.050 & 8.065 & 5.065e-11 & 7.743e-7 & 502.465 & 9.42e-9 \\
Bisection & 1.5 & 36.66 & 31.62 & 5.061e-11 & 7.786e-10 & 502.465 & 1.946e-13 \\
Projected Newton & 1.5 & 4.220 & 1.936 & 0.03598 & 0.0001776 & 502.465 & 1.548e-7 \\
\\
Naive & 4.0 & 1.000 & 0.1429 & 5.181e4 & 10740.0 & 517.692 & -2.316e-15 \\
Dual Newton & 4.0 & 4.880 & 6.691 & 1.117e-10 & 6.446e-5 & 488.994 & 2.556e-9 \\
Bisection & 4.0 & 61.59 & 42.34 & 1.111e-10 & 3.132e-10 & 488.994 & -1.36e-15 \\
Projected Newton & 4.0 & 6.420 & 2.476 & 0.01578 & 0.07242 & 488.993 & 1.232e-7 \\
\\
Naive & 10.0 & 1.000 & 0.1444 & 8.506e5 & 119200.0 & 505.418 & -2.421e-14 \\
Dual Newton & 10.0 & 6.870 & 13.41 & 7.165e-9 & 0.0003264 & 405.299 & 1.379e-9 \\
Bisection & 10.0 & 216.6 & 284.9 & 7.154e-9 & 4.595e-5 & 405.299 & -2.359e-15 \\
Projected Newton & 10.0 & 10.77 & 3.931 & 0.1228 & 0.5518 & 405.299 & 1.398e-7 \\
\\
Naive & 99.0 & 1.000 & 0.1475 & 2.318e6 & 4.908e5 & 248.933 & 4.453e-15 \\
Dual Newton & 99.0 & 12.03 & 12.78 & 4.423e-8 & 0.004281 & 225.815 & 4.045e-9 \\
Bisection & 99.0 & 218.5 & 208.7 & 4.431e-8 & 0.000912 & 225.815 & 4.796e-16 \\
Projected Newton & 99.0 & 13.18 & 3.775 & 0.1419 & 0.1335 & 225.814 & 1.563e-5 \\
\\
Naive & 100.0 & 1.000 & 0.1459 & 8.748e5 & 5.165e5 & 221.562 & -4.091e-15 \\
Dual Newton & 100.0 & 13.44 & 10.42 & 1.732e-8 & 0.001485 & 196.074 & 8.638e-10 \\
Bisection & 100.0 & 199.8 & 221.7 & 1.867e-8 & 0.0008292 & 196.074 & -3.126e-15 \\
Projected Newton & 100.0 & 0.000 & 0.008133 & 1.32e12 & 1.8970e4 & 176.687 & 0.08643 
\end{tabular}
\end{center}
\end{table}

 \begin{table}[tbh]
\caption{Average performance of $\ell_p$-ball projection algorithms for $d=1,000,000$ ($p < 1$)}\label{table_noncvx} 
\begin{center}
\tiny
\begin{tabular}{lcccccccr}
Method & $p$ & Iters & Secs & KKT1 & Duality gap & Obj & $\frac{1}{r}\|\mathbf{x}^{\star}\|_p - 1$ & \% NaN \\
\hline
Naive & 0.1 & 1.0 & 0.49 & 0.0 & \text{NaN} & 34.24 & 3.611e-14 & 0 \\
Bisection & 0.1 & 243.9 & 43.64 & 0.0 & 4.874e-119 & 34.24 & -4.147e-7 & 100 \\
IRBP1 & 0.1 & 1001.0 & 198.6 & 0.0 & \text{NaN} & 999.9 & -1.0 & 0 \\
IRBP2 & 0.1 & 951.1 & 408.0 & 0.0 & 0.0 & 4.034e-14 & 5.686 & 100 \\
IRBP3 & 0.1 & 1.0 & 4.641 & 0.0 & \text{NaN} & 999.9 & -1.0 & 0 \\
\\
Naive & 0.3 & 1.0 & 0.4496 & 0.0 & \text{NaN} & 157.7 & -1.018e-14 & 0 \\
Bisection & 0.3 & 106.3 & 23.49 & 0.0 & 1.112e-38 & 157.7 & -3.027e-8 & 100 \\
IRBP1 & 0.3 & 1001.0 & 208.7 & 0.0 & \text{NaN} & 999.9 & -1.0 & 0 \\
IRBP2 & 0.3 & 1001.0 & 440.0 & 0.0 & 0.0 & 4.148e-14 & 5.686 & 100 \\
IRBP3 & 0.3 & 1.0 & 5.742 & 0.0 & \text{NaN} & 999.9 & -1.0 & 0 \\
\\
Naive & 0.5 & 1.0 & 0.4509 & 4.541e-12 & \text{NaN} & 286.5 & -6.459e-15 & 0 \\
Bisection & 0.5 & 77.04 & 19.17 & 9.371e-8 & 6.298e-22 & 281.7 & 4.692e-8 & 100 \\
IRBP1 & 0.5 & 804.7 & 229.2 & 3.569e-14 & 4.328e-16 & 363.0 & -0.02082 & 100 \\
IRBP2 & 0.5 & 809.2 & 327.1 & 7.591e-13 & 4.335e-16 & 362.9 & -0.02115 & 10 \\
\\
Naive & 0.7 & 1.0 & 0.2217 & 0.0137 & 2.084e-9 & 762.3 & 3.197e-15 & 100 \\
Bisection & 0.7 & 52.3 & 20.94 & 1.73e-12 & 2.514e-15 & 365.1 & -1.378e-8 & 100 \\
IRBP1 & 0.7 & 8.37 & 2.355 & 1.428e-5 & 8.791e-10 & 417.4 & -0.008257 & 100 \\
IRBP2 & 0.7 & 8.37 & 3.099 & 1.425e-5 & 8.806e-10 & 417.5 & -0.008243 & 100 \\
IRBP3 & 0.7 & 7.03 & 113.9 & 1.432e-5 & 9.549e-10 & 417.9 & -0.00304 & 100 \\
\\
Naive & 0.9 & 1.0 & 0.2242 & 0.4134 & 2.246e-7 & 651.6 & -8.297e-15 & 100 \\
Bisection & 0.9 & 35.71 & 15.72 & 2.438e-14 & -1.401e-11 & 429.1 & 3.039e-7 & 100 \\
IRBP1 & 0.9 & 8.36 & 2.378 & 0.0007969 & 9.742e-7 & 433.9 & -0.005977 & 100 \\
IRBP2 & 0.9 & 8.36 & 3.082 & 0.0007997 & 1.108e-6 & 433.9 & -0.00598 & 100 \\
IRBP3 & 0.9 & 6.43 & 77.27 & 0.001567 & 1.085e-6 & 433.2 & -0.004778 & 100 \\
\\
Naive & 0.99 & 1.0 & 0.2113 & 0.123 & 4.541e-8 & 457.2 & 8.138e-16 & 100 \\
Bisection & 0.99 & 27.52 & 13.22 & 7.071e-15 & -1.258e-14 & 448.9 & 5.677e-9 & 100 \\
IRBP1 & 0.99 & 9.0 & 2.618 & 0.0004196 & 1.962e-8 & 449.3 & -0.0007902 & 100 \\
IRBP2 & 0.99 & 9.0 & 3.46 & 0.0004194 & 1.954e-8 & 449.3 & -0.0007902 & 100 \\
IRBP3 & 0.99 & 7.0 & 83.49 & 0.0005736 & 9.499e-9 & 449.0 & -0.0001324 & 100 \\
\end{tabular}
\end{center}
\end{table}

 \subsection{Compressed sensing}

Having observed the effectiveness of the bisection approach for projection onto non-convex norm balls, in this section we consider its application to compressed sensing, i.e., recovery of a sparse signal from linear measurements. 
Suppose we want to estimate an unknown but sparse signal $\bs \in \mathbb{R}^d$ from $m$ noisy observations $\bb \in \mathbb{R}^m$ through measurement or sensing matrix $\bA \in \mathbb{R}^{m\times d}$ such that
\[
	\bb = \bA\bs + \bepsilon
	,
\]
where $\bepsilon \in \mathbb{R}^m$ is the noise.
A possible approach is to solve the $\ell_p$-constrained least squares problem
\begin{equation}\label{eqn:cs}
	\min_{\bx \in \mathbb{R}^d} \frac{1}{2}\|\bb - \bA\bx\|_2^2
	\quad
	\text{subject to}
	\quad
	\|\bx\|_p \le r
\end{equation}
for $p \in [0, 1]$.
It is well known, especially when $p=1$,
that under certain conditions on the sensing matrix $\bA$ the solution $\bx^{\star}$ to problem \eqref{eqn:cs} is equal to $\bs$ with high probability \citep{donoho2006}.

Since the $\ell_1$-ball is the convex hull of the $\ell_0$-ball that exactly quantifies sparsity, use of $\ell_p$ norms with $p<1$ is expected to recover $\bs$ better than $\ell_1$ norm, as evidenced by \citet{chartrand2008restricted,blumensath2009,chartrand2016}.
Since in this case problem \eqref{eqn:cs} is non-convex, its global optimum is difficult to find. However, the sequence generated by the projected gradient descent (PGD) method to approximately solve problem \eqref{eqn:cs}
\begin{equation}\label{eqn:pdg}
	\bx^{k+1} = P_{rB_p}[\bx^k + \gamma_k\bA^T(\bb - \bA\bx^k)]
\end{equation}
with $\bx^0 = \mathbf{0}$
has been shown to perform well \citep{bahmani2013,blumensath2009}. Here $\gamma_k > 0$ is the step size at iteration $k$.
Bahmani and Raj \citep{bahmani2013} analyzed the rate of convergence of $\bx^k$ to $\bs$ as a function of $p$,
showing that the sufficient conditions for exact signal recovery become more stringent while robustness to noise and convergence rate worsen, as $p$ increases from $0$ to $1$.
Oymak et al. \citep{oymak2018,sattar2020} extended the analysis for more general (non-convex) constraint sets including $\ell_p$-balls. Their experiments compared cases $p=0$, $0.5$, and $1$, and suggest that while $p=0$ outperforms $p=1$, it is dominated by $p=0.5$.

Our bisection approach for computing $P_{rB_p}$ opens up the opportunity of fully assessing the performance of PGD \eqref{eqn:pdg} for various values of $p$. Following \citet[Sect. III-A]{oymak2018}, we fix the dimension $d=1000$ and vary the sparsity level $s$ from 50 to 1000 (incremented by 50) and the number of measurements $m$ from 200 to 6000 (incremented by 200).
We consider a sparse signal $\bs$ whose support (of size $s$) is chosen uniformly at random with i.i.d. standard normal values, and a random sensing matrix $\bA^{m \times d}$ whose entries are i.i.d. standard normal. Noiseless measurements ($\bepsilon=\mathbf{0}$) are assumed. 
The radius $r$ of the $\ell_p$-ball is set to $\|\bs\|_p$.\footnote{This optimal tuning parameter as required by the theory of \citet{oymak2018} can be relaxed. However, we closely follow the experiment setup of \citet{oymak2018} here.}
A PGD trial with $\gamma_k=1/m$ is stopped after 500 iterations, and recovery is declared successful if $\|\bx^{(n)} - \bs\|_2 / \|\bs\|_2 < 10^{-3}$ to set the optimal solution $\hat{\bx}=\bx^{(n)}$. The average success rate of 50 trials for each combination of $m$ and $s$ is recorded.

The result, plotted in Figure \ref{fig:cs}, clearly demonstrates the phase transition phenomenon in compressed sensing; namely, for each sparsity level, there is a sharp transition of the success probability as the number of measurements increases. The success rate is higher if the signal is more sparse.
This result also confirms the finding of \citet{oymak2018} that $p=0.5$ outperforms both $p=1$ and $p=0$. Among the latter two, $p=0$ has a higher probability of success.
It is interesting to note that both $p=0.1$ and $0.9$ perform slightly better than $p=0$, and that $p=0.3$ and $0.7$ perform similarly to $p=0.5$.
Hence unlike what is predicted by the theory of \citet{bahmani2013}, there seems a \emph{range} of intermediate values of $p$ away from both $0$ and $1$ that performs best in combination with PGD. The reason for this will be an interesting subject of further research.

\begin{figure}
\centering
\begin{tabular}{cc}
\includegraphics[width=0.45\textwidth]{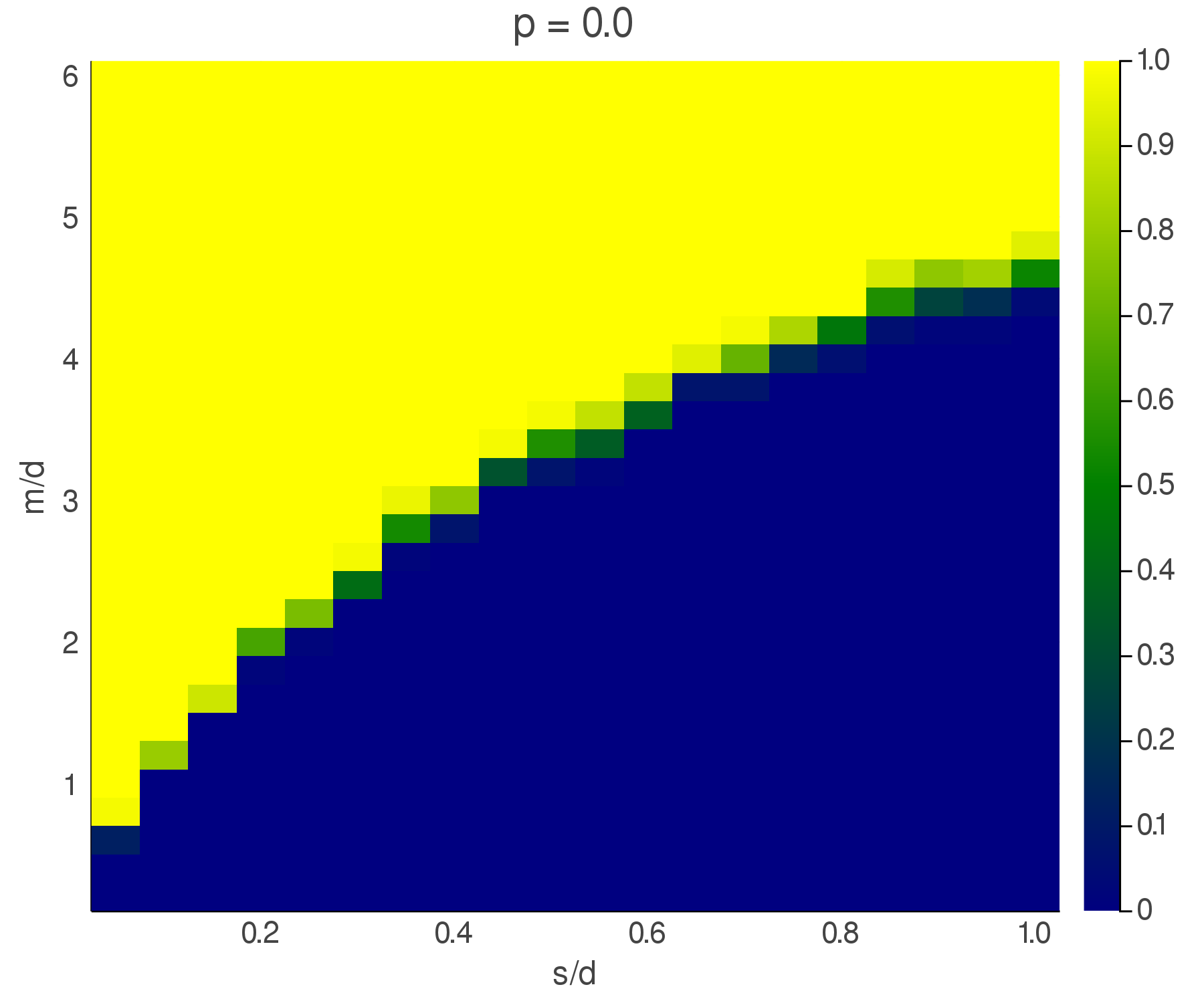} 
&
\includegraphics[width=0.45\textwidth]{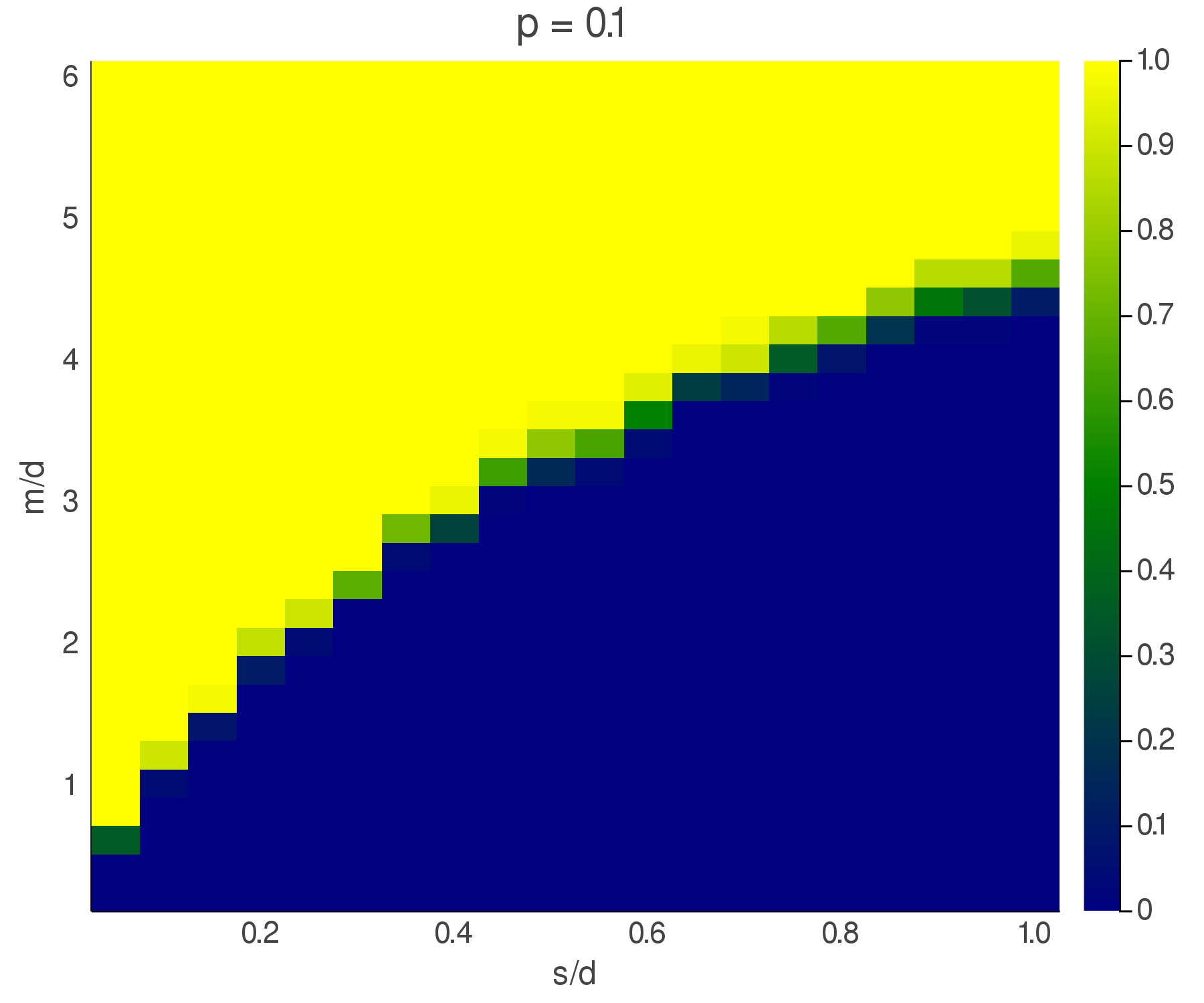}
\\
\includegraphics[width=0.45\textwidth]{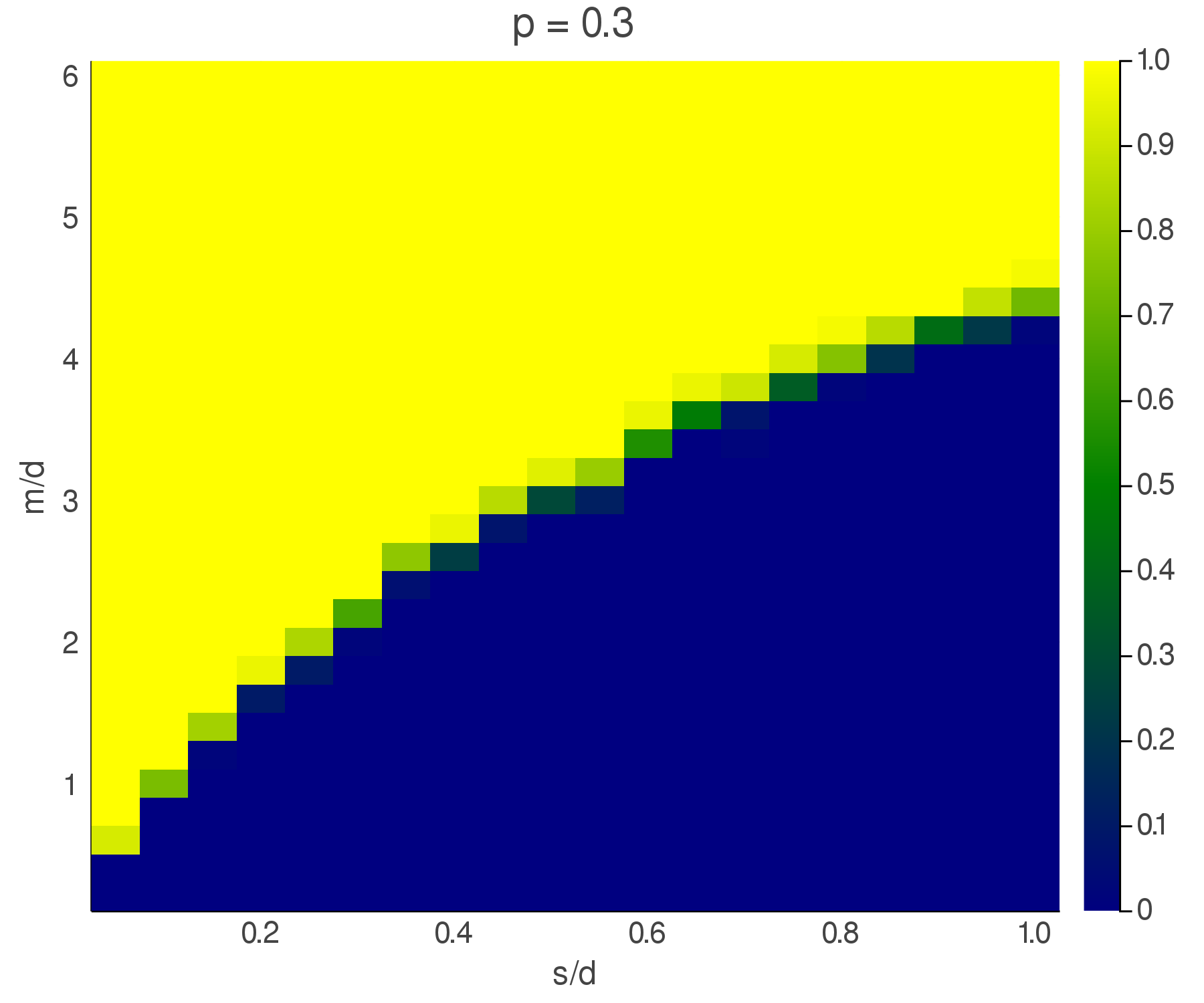} 
&
\includegraphics[width=0.45\textwidth]{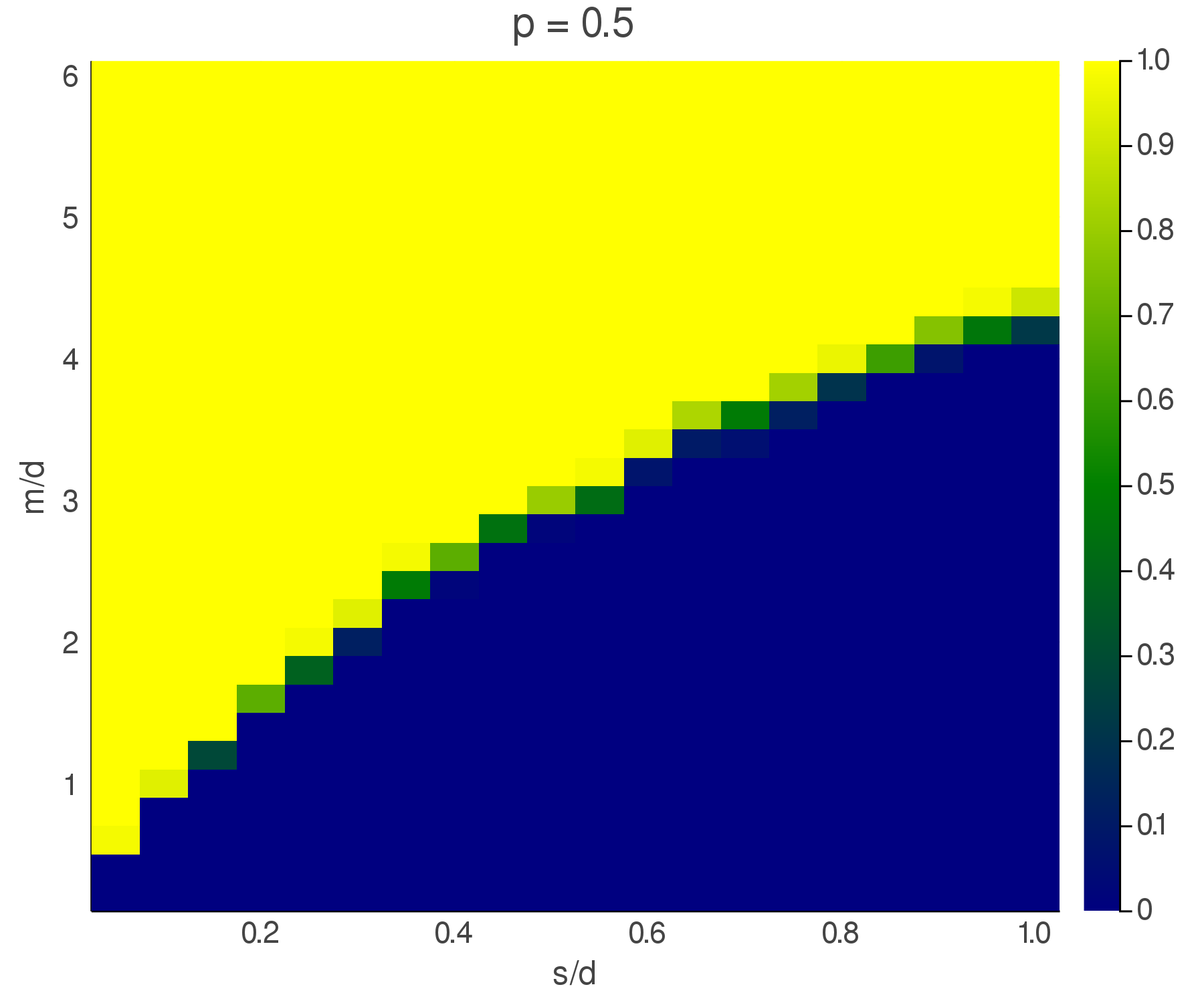}
\\
\includegraphics[width=0.45\textwidth]{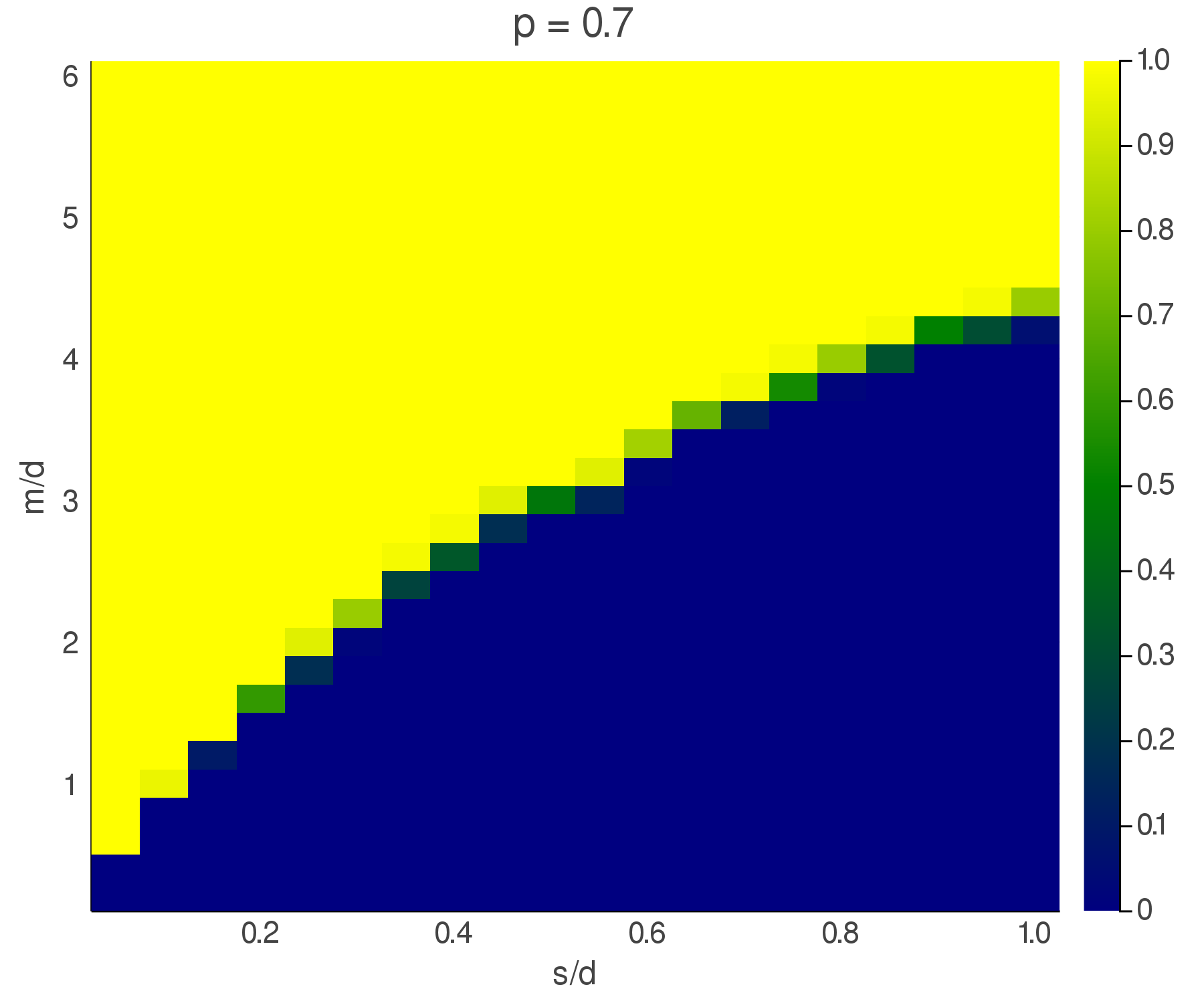} 
&
\includegraphics[width=0.45\textwidth]{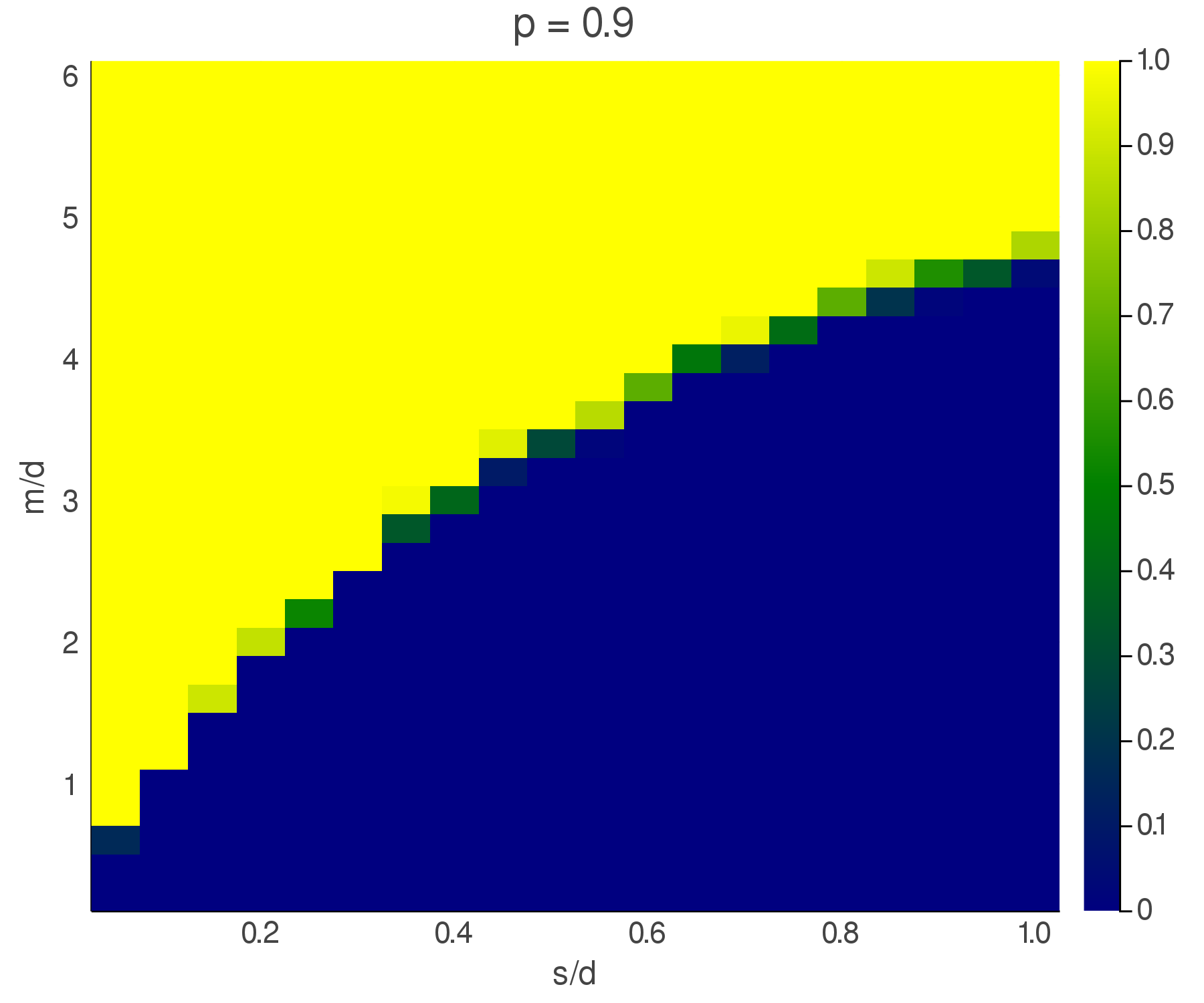}
\\
\includegraphics[width=0.45\textwidth]{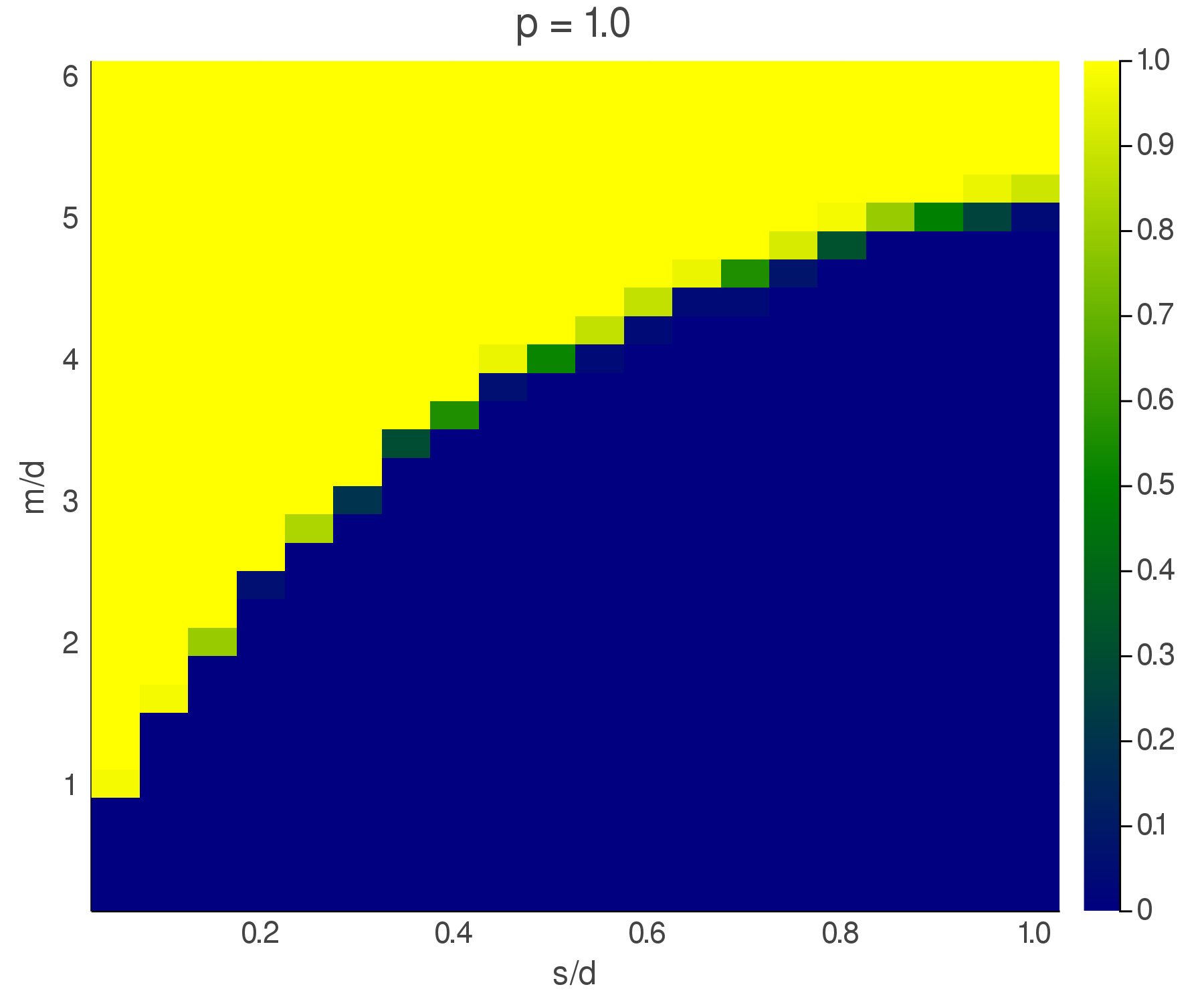} 
&
~
\end{tabular}
\caption{Phase transition diagram of $\ell_p$-projected gradient descent. The bright, yellow regions represent the combinations of the number of measurement $m$ and sparsity level $s$ that result in exact signal recovery.}
\label{fig:cs}
\end{figure}

 \section{Discussion}\label{sec:discussion}
We have proposed robust and highly scalable algorithms for projecting onto $\ell_p$ balls in general, a key component of many learning tasks. Their merits are demonstrated empirically and agree with our theoretical treatment; our contributions outpace and outperform the limited prior work for a difficult but core computational problem, and provide a unified view of the convex and non-convex cases. These tools open the door to previously intractable penalty and constraint formulations, which have shown to be often better suited to various learning tasks than their more convenient counterparts.

\iffalse
In the case $p<1$, both Newton method fail miserably, except for the dual Newton when $p$ is very close to $1$. 
The na{\"i}ve method performs relatively well 
at the opposite end of the interval $(0,1)$. 
Bisection performs well 
near the endpoints with very small duality gaps, 
but falters 
in the middle of the interval.
Perhaps, 
these deficiencies are inevitable for such a difficult nonconvex problem.
\fi
 
\onehalfspacing

\footnotesize
\singlespacing

\bibliographystyle{chicago}
\bibliography{lp_proj}

\end{document}